\documentclass[a4paper,11pt,twoside]{amsart}

\usepackage[latin1]{inputenc}
\usepackage{amsmath, amsfonts, amssymb,amsthm}
\usepackage[mathscr]{eucal} %Proporciona el comando
\usepackage[pdftex]{graphicx}
\usepackage{color}
\usepackage[T1]{fontenc}
\usepackage[sc]{mathpazo}
\usepackage[all]{xy}
\usepackage{hyperref}
\usepackage[flenq]{nccmath}

\usepackage {anysize}
\marginsize{2cm}{2cm}{2cm}{3cm}
\usepackage{enumerate}

\definecolor{alert}{rgb}{0.8,0,0}

\newcommand{\s}{{\bf \mathrm{S}}}
\newcommand{\h}{\mathbb{H}^2}
\newcommand{\hr}{\mathbb{H}^2\times\real}
\renewcommand{\div}[1]{\mathrm{div}\left( #1 \right)}
\newcommand{\area}{\mathcal{A}}
\newcommand{\D}{\mathcal{D}}
\newcommand{\hor}{\mathcal{H}}
\newcommand{\pol}{\mathcal{P}}
\newcommand{\poldo}{\mathcal{{\bf P}}}

\newcommand{\real}{\mathbb{R}}

\newcommand{\abs}[1]{\left\lvert #1 \right\rvert}
\newcommand{\polprime}{\mathcal{P}^\prime}
\newcommand{\polprimedo}{\mathcal{{\bf P^\prime}}}
\newcommand{\alphad}{\alpha_{\mathcal{D}}}
\newcommand{\alphadzero}{\alpha_{\mathcal{D}_0}}
\newcommand{\alphadtau}{\alpha_{\mathcal{D}_\tau}}
\newcommand{\absc}[1]{\left\lvert \left[  #1 \right]\right\rvert}
\newcommand{\betadzero}{\beta_{\mathcal{D}_0}}
\newcommand{\betadtau}{\beta_{\mathcal{D}_\tau}}
\newcommand{\betad}{\beta_{\mathcal{D}}}
\newcommand{\poltio}{\widetilde{\pol}}
\newcommand{\poltiodo}{\widetilde{\mathcal{{\bf P}}}}
\newcommand{\poltprimedo}{\mathcal{{\bf P^{\prime\prime}}}}
\newcommand{\poltprime}{\mathcal{P}^{\prime\prime}}
\newcommand{\polttio}{\widetilde{\pol}^\prime}
\newcommand{\polttiodo}{\mathcal{{\bf \widetilde{P}^{\prime}}}}
\newcommand{\polttiodot}{\mathcal{{\bf \widetilde{P}}}^{\prime\ T}}
\newcommand{\dtau}{\mathcal{D}_{\tau}}

\newtheorem{theorem}{Theorem}[section]
\newtheorem{proposition}[theorem]{Proposition}

\newtheorem{lemma}[theorem]{Lemma}
\newtheorem{claim}[theorem]{Claim}

\theoremstyle{definition}
\newtheorem{definition}[theorem]{Definition}
\theoremstyle{remark}
\newtheorem{remark}[theorem]{Remark}

\numberwithin{equation}{section}

\title[Entire Constant Mean Curvature Graphs in $\mathbb{H}^2\times\mathbb{R}$]{Entire Constant Mean Curvature Graphs in $\mathbb{H}^2\times\mathbb{R}$}

\author[Abigail Folha]{Abigail Folha}
\address{Instituto de Matem\'atica e Estat\'istica -Universidade Federal Fluminense. Rua Prof. Marcos Waldemar de Freitas Reis, S/N, Niter\'oi-RJ, 24210-201. Brazil }
\email{abigailfolha@id.uff.br}

\author[Harold Rosenberg]{Harold Rosenberg}
\address{Instituto Nacional de Matem\'atica Pura e Aplicada. Estrada Dona Castorina 110, Rio de Janeiro 22460-320. Brazil}
\email{rosen@impa.br}

\thanks{}

\subjclass[]{}


\begin{document}

\begin{abstract}
For $0\leq H< 1/2$, we construct entire $H$-graphs in $\hr$ that are parabolic and not invariant by one parameter groups of isometries of $\hr$. Their asymptotic boundaries are $(\partial_\infty\mathbb{H}^2)\times\mathbb{R}$;  they are dense at infinity. When $H=0$ the examples are minimal graphs constructed by P. Collin and the second author \cite{CR}.
 \end{abstract}

\maketitle

\section{Introduction}

The study of complete H-graphs is well understood when they are invariant surfaces, i.e., invariant by one parameter isometry groups. R. S\'a Earp and E. Toubiana have written many papers describing these surfaces in great detail \cite{ST1}, \cite{S}. Also, J.M. Manzano and B. Nelli have analysed the area growth of these examples, \cite{MN}; particularly those that are complete graphs over ideal domains of $\mathbb{H}^2$; called H-Scherk graphs.  When $H=0$ these graphs were shown to exist in \cite{CR} and for $0<H<1/2$, necessary and sufficient conditions on ideal domains of $\mathbb{H}^2$ for their existence were obtained by S. Melo and the first author \cite{FM}. This was an important extension of the theorem proved in \cite{HRS2}  for compact domains of $\mathbb{H}^2$ and $\mathbb{S}^2$.

Consider entire H-graphs over $\mathbb{H}^2$ in $\hr$. When $H=0$ one can solve a Plateau problem at infinity to obtain such a graph. More precisely, consider a continuous graph $C$ over $\partial_\infty\mathbb{H}^2$ in $\partial_\infty\hr$. Then, B. Nelli and the second author proved there is a unique entire minimal graph in $\hr$ with asymptotic boundary $C$, \cite{NR}.

When $0<H\leq 1/2$ one can not solve such a Plateau problem. For these values of $H$, there exist entire H-graphs with one minimum and diverging to infinity as one goes to infinity in $\mathbb{H}^2$. They are rotationally invariant and described in \cite{AR}, \cite{ST1}. Hence an entire H-graph whose asymptotic boundary contains an arc which is a graph over a non trivial arc of $\partial_\infty\mathbb{H}^2$ cannot exist for $0<H\leq 1/2$ by the maximum principle.

Since there are compact H-spheres for $H>1/2$, there are no entire H-graphs for $H>1/2$.

%%%%%%%%%%%%%%%%%

The theory of $H = 1/2$ surfaces is different than that of other values of $H$.  For a complete $H = 1/2$ surface in $\hr$,  I. Fern\'andez and P. Mira  have introduced a Gauss map on the surface taking values in $\mathbb{H}^2$ \cite{FeMi}. They develop a very interesting study of such surfaces in terms of this Gauss map, which they show is a harmonic map.  They also formulate a Plateau problem in terms of this Gauss map.  This is not the Plateau problem we refer to in the previous paragraph.  

%%%%%%%%%%%%%%%%%%%%%%%%%%%
%%%%%%%%%%%%%%%%%%%%%%%%%%

One knows that a complete $H=1/2$ surface immersed in $\hr$, that is transverse to the vertical is an entire vertical graph \cite{HRS}.  Also in   \cite{HRS}, they explain how the work of  Fern\'andez-Mira \cite{FeMi}, and  Wan-Au \cite{WA} and Wan \cite{W} prove that given a holomorphic quadratic differential on $\mathbb{C}$ or the disc, there is a  complete $H=1/2$ surface immersed in $\hr$, that is transverse to the vertical, hence an entire vertical graph.  The quadratic differential  $Q$ on the graph is the Abresch-Rosenberg holomorphic quadratic differential of the graph.

The H-surfaces in $\hr$  with vanishing holomorphic quadratic differential are invariant surfaces \cite{AR}, but, in general, invariant H-surfaces do not have this property; for instance, any $H=1/2$ surface in $\hr$ that is not a horocylinder or a rotational surface that meets its rotation axis orthogonally has non-zero holomorphic differential; this follows from the classification in \cite{AR}. Thus starting with a non-zero holomorphic quadratic differential on $\mathbb{C}$, the correspondence of I. Fern\'andez and P. Mira may produce an entire $H=1/2$ surface in $\hr$  that is an invariant surface. However, almost all such $Q$ on $\mathbb{C}$ produce non-invariant $H= 1/2$ entire graphs. We are grateful to Pablo Mira for explaining why almost all such $Q$ on $\mathbb{C}$ produce non-invariant $H= 1/2$ entire graphs.  Here is his argument. The space of invariant $H=1/2$ surfaces in $\hr$ depends on a finite number of real parameters (the possible choices of the one or two-parameter isometry subgroup with respect to which the surface is invariant, and the initial conditions for the corresponding ODE). But the space of entire $H=1/2$ graphs is parametrized by the space of holomorphic quadratic differentials on $\mathbb{C}$ or the disc, which is much larger. So, most entire $H=1/2$ graphs are not invariant surfaces.

%%%%%%%%%%%%%%%%%%%%%%%%%%%%%%%
%%%%%%%%%%%%%%%%%%%%%%%%%%%%%%

  Previously, the only known examples of entire H- graphs in 
$\hr$, when $0 < H < 1/2$, are the invariant examples.  These examples are conformally the disc \cite{AR}.

%%%%%%%%%%%%%%%%

In this paper we construct entire H-graphs in $\hr$ that are not invariant surfaces for $0<H<1/2$. Their asymptotic boundary is $\partial_\infty\hr$. Moreover they are conformally $\mathbb{C}$.  The idea of the proof is to follow  the construction of such a graph for $H=0$ by \cite{CR}. Here the  details are considerably more complicated because of Jenkins-Serrin conditions on an ideal domain, for the existence of such a H-Scherk graph, involve the area of admissible curved polygons in the domain; not only the side lengths.     

\section{Preliminaries}

 The aim of this section is to fix some notations and state an existence theorem for graphs having constant mean curvature over unbounded domains in $\hr$.
 
 Let $\Omega\subset\h$ be a domain and $\Sigma\subset\hr$ be a graph over $\Omega$ of a map $u:\Omega\longmapsto\real$.  Assume  $\Sigma$ has constant mean curvature $H$, then $u$ satisfies the following equation
 
 \begin{equation}\label{div}
 \div {\frac{\nabla u}{\sqrt{1+\abs{\nabla u }^2}} }=2H,
\end{equation}  
where the divergence and gradient are taken with respect to the metric of $\mathbb{H}^2$. A function $u$ satisfying \eqref{div} is called a \emph{solution} of \eqref{div} in $\Omega$. 

\begin{definition}[An ideal domain] Let $\D\subset \h$ be  an ideal simply connected domain  whose boundary $\partial\D\subset \h$ is composed of  arcs  $\{A_j\}, j=1,\cdots, n,$ and $\{ B_l\}, l=1,\cdots, n,$ that satisfy $\kappa(A_j)=2H$ and $\kappa(B_l)=-2H$ with respect to the interior of $\D$, for some $0<H<1/2$. The asymptotic boundary of $\D$, $\partial_\infty\D=\{a_i\}, \ i=0, \cdots, 2n$ with $a_0=a_{2n}$,  is composed of vertices $a_i$, the end points of the boundary curves $A_i$ and $B_i$. Assume that no two  arcs $A_i$ and no two arcs $B_i$ have a common endpoint. Moreover, all vertices of $\D$ are in the asymptotic boundary of $\h$.  Note that the area of such a domain is finite. 
\end{definition}

\begin{figure}[ht]%
\centering
\includegraphics[width=3.0in]{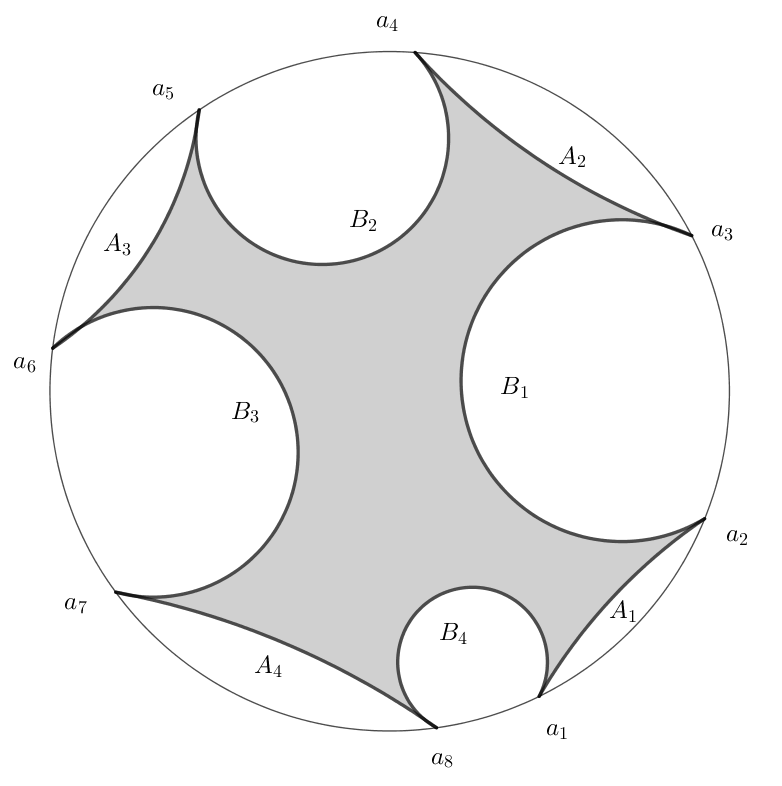}
\includegraphics[width=3.0in]{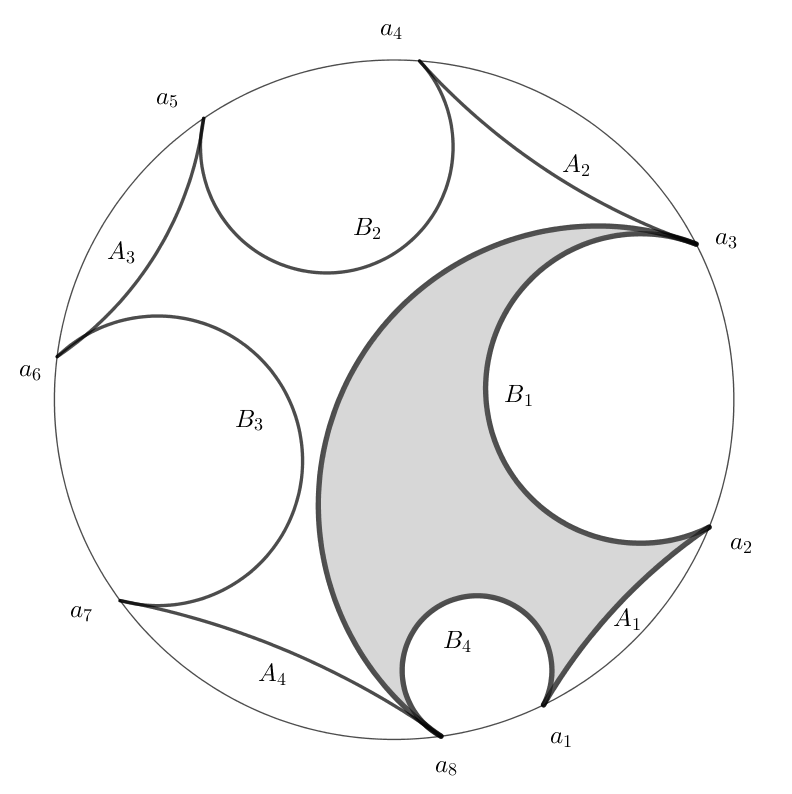}
\caption{An ideal domain and an ideal domain with an inscribed curved polygon.}
\label{fig0d}
\end{figure}

\begin{definition}[Ideal curved polygon] Fix an $H, 0<H<1/2$. An ideal curved polygon in $\h$ is a set composed of a finite number of vertices, at least one of them is in  $\partial_\infty\h$, the  asymptotic boundary of $\h$, and equidistant curves joining these vertices having curvature $\pm 2H, 0<H<1/2$.   We do not require the signs of the curvature of the equidistant arcs to alternate as one traverses the polygon. 
\end{definition}

\begin{definition}[Inscribed curved polygons] Let $\D$ be an ideal domain in $\mathbb{H}^2$. We say that $\pol\subset (\D\cup\partial\D)$ is an inscribed curved polygon if it is an ideal curved polygon, and all vertices of $\pol$ are vertices of $\partial\D$.
\end{definition}

%\begin{figure}[ht]%
%\centering
%\includegraphics[width=3.5in]{domain.pdf}
%\caption{}%{Figure 1}
%\label{fig0}
%\end{figure}

The existence of graphs having constant mean curvature defined over an ideal  domain depends on the "length" of the edges $\{A_i\}$ and $\{B_i\}$ and certain areas. Once these lengths are infinite, we proceed as follows.  Let $\D$ be an ideal domain and let $\{a_i\},\ i=1,2,\cdots, 2n$, be the vertices of $\partial\D$. At each $a_i$, place a horocycle $\hor_i$  such that $\hor_i\cap\hor_j=\emptyset$, for all $i\neq j$.  We let $F_i$ be the horodisk bounded by $\hor_i$.  Each $A_i$ meets exactly two horodisks $\{F_i\}$,  we denote by $\widetilde{A}_i$ the compact part  of $A_i$ outside these horodisks. We define $\abs{A_i}$ to be the length of $\widetilde{A}_i$. We define $\abs{\eta_j}$ and $\widetilde{\eta}_j$ in the same way for any  ideal arc contained in $\D\cup\partial\D$. These lengths depend on the choice of horocycles.

Now we fix some notation. Given an ideal domain $\D$ and  an inscribed curved polygon $\pol=\displaystyle\bigcup_{j}\eta_j\ \ $ we set

\begin{enumerate}
[$\bullet$]
\item $\alpha_{\D}(\pol)=\displaystyle\sum_{A_i\subset\pol}\abs{A_i}$,\ \ \  $\beta_{\D}(\pol)=\displaystyle\sum_{B_i\subset\pol}\abs{B_i}$ and $l(\pol)=\displaystyle\sum_{\eta_j\subset\pol}\abs{\eta_j}$
\item $a_1a_2\cdots a_n$ will denote a piecewise smooth Jordan curve in $\mathbb{H}^2$ with vertices $a_1,a_2, \cdots, a_n$ and smooth arcs joining $a_i$ to $a_{i+1}$, $1<i<n-1$ and a smooth arc joining $a_n$ to $a_1$. The vertices $a_i$ may be ideal vertices, i.e., $a_i\in\partial_\infty\mathbb{H}^2$. We orient $a_1a_2\cdots a_n$ counter-clockwise.
\item $\poldo$ is the domain bounded by $\pol$ and $\area(\poldo)$ its area.  
\item Given an ideal curved polygon $\pol$ and two points $p$ and $q$ in one edge of $\pol$ we denote by $[pq]$ the arc in $\pol$ joining $p$ to $q$.
\item $\D^T:=\D\setminus\{\cup_{i}F_i\}$ is the domain truncated by $\{\hor_i\}$.   
 
\end{enumerate}

From now on a graph having constant mean curvature will be denoted by H-graph. Now we are ready to state an existence theorem for H-graphs defined over unbounded domains having infinite boundary values, for $0<H<1/2$.

\begin{theorem} 
 [\cite{FM} Theorem 3.1] \label{teoexist} We fix $H, \ 0<H<1/2$. Let $\D$ be an ideal domain whose boundary is an ideal curved polygon composed of arcs $\{A_i\}$ and $\{B_i\}$ having curvature $\kappa(A_i)=2H$ and $\kappa(B_i)=-2H$ with respect to the domain $\D$. There is a solution $u:\D\longrightarrow\real$ of \eqref{div}, which assumes values $+\infty$ on each arc $A_i$ and $-\infty$ on each $B_i$ if, and only if,  the following two conditions are satisfied
\begin{equation}\label{cond1}
 \alpha_{\D}(\partial\D)-\beta_{\D}(\partial(\D))=2H\area(\D), 
\end{equation}
 and for all inscribed curved polygons $\pol$,
\begin{equation}\label{cond2}
2\alpha_{\D}(\pol) < l(\pol)+2H\area(\poldo) \ \ \textnormal{and} \ \ 2\beta_{\D}(\pol) < l(\pol)-2H\area(\poldo).
\end{equation} 
 
\end{theorem} 
\begin{definition}[Ideal admissible domain]
A domain $\D$ satisfying the conditions of  Theorem \ref{teoexist} is called an ideal admissible domain.
\end{definition}

\begin{remark}
The conditions \eqref{cond1} and \eqref{cond2} do not depend on the choice of horocycles, i.e. if they hold for one choice of horocycles then they hold for all choices of smaller horocycles.
 \end{remark}

\begin{definition}
Let $\D$ be an ideal admissible domain whose boundary is an ideal curved polygon composed of arcs $\{A_i\}$ and $\{B_i\}$ having curvature $\kappa(A_i)=2H$ and $\kappa(B_i)=-2H$ with respect to the domain $\D$. A solution $u$ of \eqref{div} having boundary values $+\infty$ on $\{A_i\}$ and $-\infty$ on $\{B_i\}$ will be  called a solution of the Dirichlet problem in $\D$.
 \end{definition}

%%%%%%%%%%%%%%%%%%%%%%%%%%%%%%%%%%%%%%%%%%
%%%%%%%%%%%%%%%%%%%%%%%%%%%%%%%%%%%%%%%%%%
%%%%%%%%%%%%%%%%%%%%%%%%%%%%%%%%%%%%%%%%%%

We now summarize the rest of this paper. Section \ref{example} will study in detail an ideal curved quadrilateral with its four vertices at infinity and edges joining the vertices of curvature $\pm 2H, \ 0<H<1/2$. Fixing three of the vertices we will analyse the positions of the fourth vertex that make the domain bounded by the quadrilateral an admissible ideal domain. Some functions arising in this section (Lemma \ref{lemma1} for example) will be used in the following sections.

In Section \ref{extension}, we start with an admissible domain $\D$ constructed in Section \ref{example} together with a solution $u$ of the Dirichlet problem on $\D$. We then attach admissible ideal quadrilaterals to the four sides of $\D$ to obtain a domain $\D_0$. We want to "extend" $u$ to a solution $u_0$ of the Dirichlet problem on $\D_0$. This does not work. First of all, $\D_0$ will not be  an admissible ideal domain so there is no solution of the Dirichlet problem on $\D_0$. We will deal with this problem by perturbing $\D_0$, we will move an ideal vertex of each added ideal quadrilateral to make $\D_0$ an admissible domain $\D_\tau$; so there are solutions on $\D_\tau$. But $u$ is $\pm\infty$ on the sides of $\D$ so extending $u$ to $\D_\tau$ doesn't make sense. However we will prove that if $K\subset \D$ is any compact then a solution $u_\tau$ on $\D_\tau$ exists that is as close to $u$ on $K$ as desired ($C^2$-close). In fact, $u_\tau$ and $\D_\tau$ will shown to exist for all $\tau>0$ sufficiently small and the estimate of $\|u_\tau-u\|_{C^2(K)}$ converges to zero as $\tau\rightarrow 0$. 

In Section \ref{conformal} we will show the conformal type of the graph of any solution to the Dirichlet problem over an ideal admissible domain is parabolic, i.e., conformally the complex plane.

Then in Section \ref{main} we iterate the extension process described in Section \ref{extension} to obtain an entire H-graph that is parabolic and whose ideal boundary is $(\partial_\infty\mathbb{H}^2)\times\real$.

%%%%%%%%%%%%%%%%%%%%%%%%%%%%%%%%%%%%%%%%%%
%%%%%%%%%%%%%%%%%%%%%%%%%%%%%%%%%%%%%%%%%%
%%%%%%%%%%%%%%%%%%%%%%%%%%%%%%%%%%%%%%%%%%

\section{An Example}\label{example}

 We consider the half-plane model of $\mathbb{H}^2$. Consider $\mu>0$ and $d_3>2\mu$ and  the  points on the asymptotic boundary of $\h$, \,  $P_1=(0,0), \ P_2=(2\mu, 0), P_3=(d_3,0)$ and $P_4=\partial_\infty\h\setminus\{y=0\}$. In this section we establish  conditions on $\mu$ and $d_3$ in order to prove the existence of an ideal admissible domain whose boundary is  a curved  quadrilateral having vertices $P_1, P_2,P_3$ and $P_4$, such that $[P_4P_1]=A_1, [P_2P_3]=A_2, [P_1P_2]=B_1$ and $[P_3P_4]=B_2$.  See the Figure \ref{fig1}.

\begin{figure}[ht]%
\centering
\includegraphics[width=4.5in]{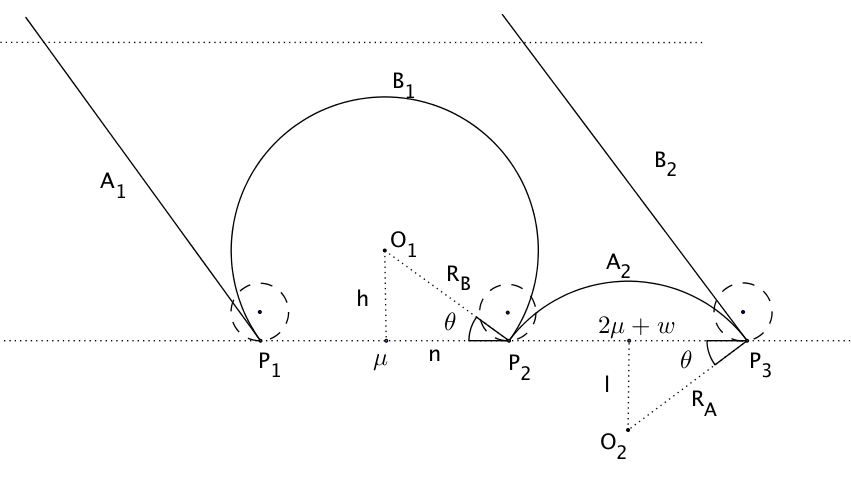}
\caption{}%{Figure 1}
\label{fig1}
\end{figure}

\begin{theorem} \label{teoquadrilatero}
For any $\mu>0,\  0<H<1/2$, and $d_3 =(1+e^{2\pi\tan\theta})2\mu$, where  $\arcsin(2H) = \theta, \theta\in [0, \pi/2]$, the quadrilateral $\D$ with vertices $P_1 = (0, 0), P_2 = (2\mu, 0), P_3 = (d_3, 0)$ and $P_4 = \partial_\infty\h\setminus\{y = 0\}$ is an ideal admissible polygon.
\end{theorem}

%\begin{proof}[Proof of Theorem \ref{teoquadrilatero}]
 The proof of Theorem \ref{teoquadrilatero}  will be long and technical.  At a first reading, one might go directly to Section \ref{extension}. 
%\end{proof}

We use the following expressions for $A_1, A_2, B_1$ and $B_2$.
\begin{eqnarray*}
%%%%%%%%%%%%%%
A_1& : & x=-y \tan\theta \\
%%%%%%%%%%%%%%%
A_2&:& (x-\dfrac{d_3+2\mu}{2})^2+(y+R_A\sin\theta)^2=R_A^2\\
%%%%%%%%%%%%%%%
B_1&:& (x-\mu)^2+(y-R_B\sin\theta)^2=R_B^2\\
%%%%%%%%%%%%%%
B_2&:&x=d_3-y \tan\theta,
\end{eqnarray*}
where $R_A:=\dfrac{d_3-2\mu}{2\cos\theta}$ and $R_B=\dfrac{\mu}{\cos\theta}$.

The domain $\D$, bounded by such a curved quadrilateral, is connected if there is no intersection between $B_1$ and $B_2$, the next claim gives the necessary condition to assure this.

\begin{claim} \label{claim1} Let $\mu>0$,  $ 0<H<1/2$, and  $\theta:=\arcsin(2H)$.  If $d_3>\dfrac{2\mu}{1-4H^2}$ then the domain $\D$ is connected.  
\end{claim} 

\begin{proof}
A point in $B_1$ satisfies the equation 
$$B_1:(x-\mu)^2+(y-\mu\tan\theta)^2=\dfrac{\mu^2 }{\cos^2\theta}.$$
And $B_2$ is a tilted line parametrized by $B_2:(d_3-y\tan\theta, y)$.

So, the intersection $B_1\cap B_2$ satisfies the equation 
\begin{equation}\label{a1}
y^2\sec^2\theta-2d_3\tan\theta\, y+d_3^2-2\mu d_3=0.
\end{equation}

The discriminant of equation \eqref{a1}  is negative if 
$$
d_3>\dfrac{2\mu }{1-\sin^2\theta} = \dfrac{2\mu }{1-4H^2}.
$$
 
 So the domain $\D$ is connected if $B_1\cap B_2$ is empty and it occurs when $d_3> \dfrac{2\mu }{1-4H^2}$, as claimed.

\end{proof}

Now we consider the followings horocycles 
\begin{eqnarray*}
\hor_1&:&x^2+(y-r)^2=r^2\\
\hor_2&:& (x-2\mu)^2+(y-r)^2=r^2\\
\hor_3&:& (x-d_3)^2+(y-r)^2=r^2\\
\hor_4&:& y=M,
\end{eqnarray*}
where $0<r<\dfrac{\mu}{2}$ is a small real number and $M>\dfrac{2\mu}{\cos\theta}$.

\begin{claim} \label{claim2} The intersections between the horocycles $\hor_1$ and the side $B_1$ and  between $\hor_2$ and $A_2$ are given by
\begin{equation*}
\hor_1\cap B_1=\{ (0,0), (x_0,y_0)\} \hspace{.4cm} \textnormal{and}\hspace{.4cm} 
\hor_2\cap A_2=\{ (2\mu,0), (x_1,y_1)\},
\end{equation*}
where
\begin{eqnarray}
\nonumber x_0&=&\dfrac{2 r \mu (r - \mu\tan\theta ) }{\mu^2\sec^2\theta+r ( r-2\mu\tan\theta)}\\
\label{a2} y_0&=&\dfrac{2 r\, R_B^2\, \cos^2\theta }{R_B^2+r^2-2r\,R_B\,\sin\theta}\\
\nonumber x_1&=& \dfrac{2( d_3^2\mu +4\mu^3 +d_3(r^2-4\mu^2 ) +d_3 r^2\cos(2\theta) ) +r \sin(2\theta) (d_3^2-4\mu^2 )}{2r^2+(d_3-2\mu )^2 +2r(r\cos(2\theta) +\sin(2\theta)(d_3-2\mu ))}\\
\label{a3} y_1&=& \dfrac{2r\, R_A^2\,\cos^2\theta}{R_A^2+r^2+2r\,R_A\,\sin\theta},
\end{eqnarray}
where $R_A=\dfrac{d_3-2\mu }{2\cos\theta}$ and $R_B=\dfrac{\mu }{\cos\theta}$. See Figure \ref{fig1}. 
\end{claim}

\begin{proof}
The proof is a straight forward computation. 

\end{proof}

Given $\mu>0$ and $H>0$, for $d_3>\dfrac{2\mu }{1-4H^2}$ we define 
\begin{equation} \label{a4} 
G(\mu, d_3,H):= \alpha(\partial \D)-\beta(\partial \D)-2H\area( \D),
\end{equation}
$G$ is well defined by Claim \ref{claim1}. We are interested in the behaviour of $G$. We place the horocycles $\hor_1,\hor_2, \hor_3$ and $\hor_4$ defined above at the vertices of $\partial \D$. Let  $R_A, R_B, y_0$ and $y_1$ be defined as in  Claim \ref{claim2}, we set $\overline{y}_1=R_A(1-\sin\theta) $, $\overline{y}_0=R_B(1+\sin\theta)$. Then, 

\begin{eqnarray}
\nonumber & &\abs{A_2}=\\
\nonumber &=&2\displaystyle\int_{y_1}^{\overline{y}_1} \dfrac{R_A}{y\sqrt{R_A^2-\left(y+R_A\sin \theta\right)^2}} \\
%%%%%%%%%%%%%%%%
\nonumber &=&-\dfrac{2}{\cos\theta}\ln\left(\dfrac{R_A^2\cos^2\theta-y R_A\sin\theta+ R_A\cos\theta\sqrt{R_A^2-(y+R_A\sin\theta)^2}}{y}\right)\displaystyle\left\vert_{y_1}
^{\overline{y}_1}\right. \\
%%%%%%%%%%%%%%%%
\nonumber &=& -\dfrac{2}{\cos\theta}\ln\left(\dfrac{R_A^2\cos^2\theta-R_A^2\sin\theta+R_A^2\sin^2\theta}{R_A(1-\sin\theta)}\right)+\\
%%%%%%%%%%%%%%
\nonumber & &+\dfrac{2}{\cos\theta}\ln\left(\dfrac{R_A^2\cos^2\theta(R_A^2+r^2+2rR_A\sin\theta)-2rR_A^3\cos^2\theta\sin\theta+R_A\cos\theta\sqrt{M_A}}{2rR_A^2\cos^2\theta}\right)\\ 
%%%%%%%%%%%%%%%%%%%%%%
\nonumber &=& -\dfrac{2}{\cos\theta}\ln\left(R_A\right)+\dfrac{2}{\cos\theta}\ln\left(\dfrac{R_A^2+r^2}{2r}+ \dfrac{1}{2r R_A\cos\theta}\sqrt{N_A}\right)\\
%%%%%%%%%%%%%%%%%%%%%%
\nonumber &=&-\dfrac{2}{\cos\theta}\ln\left(R_A\right)+\dfrac{2}{\cos\theta}\ln\left( \dfrac{R_A^2+r^2}{2r}+ \dfrac{1}{2r R_A\cos\theta}\sqrt{R_A^6\cos^2\theta+R_A^2r^4\cos^2\theta-2r^2R_A^4\cos^2\theta}\right)\\
%%%%%%%%%%%%%%%%%%%%%%
\nonumber &=&-\dfrac{2}{\cos\theta}\ln\left(R_A\right)+\dfrac{2}{\cos\theta}\ln\left( \dfrac{R_A^2+r^2}{2r}+ \dfrac{1}{2r}\sqrt{(R_A^2-r^2)^2}\right)\\
%%%%%%%%%%%%%%%%%%%%%%
\nonumber &=&\dfrac{2}{\cos\theta} \ln\left(\dfrac{R_A}{r}\right),
\end{eqnarray}
where $M_A:=R_A^2(R_A^2+r^2+2rR_A\sin\theta)^2-(2rR_A^2\cos^2\theta+R_A\sin\theta(R_A^2+r^2+2rR_A\sin\theta))^2 $ and \\ $N_A:=-(4r^2R_A^2+R_A^6\sin^2\theta+ R_A^2 r^4\sin^2\theta +4r R_A^5\sin\theta+2R_A^4r^2\sin^2\theta+4r^3R_A^3\sin\theta)+R_A^6+r^4R_A^2+2r^2R_A^4+4r^2R_A^4\sin^2\theta+4rR_A^5\sin\theta+4r^3R_A^3\sin\theta$

\begin{eqnarray}
\nonumber & &\abs{B_1}=\\
\nonumber &=& 2\int_{y_0}^{\overline{y_0}}\dfrac{R_B}{y\sqrt{R_B^2-(y-R_B\sin\theta)^2}}\\
%%%%%%%%%%%%%%%%
\nonumber &=&\dfrac{2}{\cos\theta}\ln\left(\dfrac{R_B^2\cos^2\theta+y R_B\sin\theta- R_B\cos\theta\sqrt{R_B^2-(y-R_B\sin\theta)^2}}{y}\right)\displaystyle\left\vert_{y_0}
^{\overline{y}_0}\right.  \\
%%%%%%%%%%%%%%%%%%%%%%%%
\nonumber &=&\dfrac{2}{\cos\theta}\ln\left(\dfrac{R_B^2\cos^2\theta+R_B^2\sin\theta+R_B^2\sin^2\theta}{R_B(1-\sin\theta)}\right) +\\
%%%%%%%%%%%%%%%%
\nonumber & & - \dfrac{2}{\cos\theta}\ln\left(\dfrac{R_B^2\cos^2\theta(R_B^2+ r^2-2rR_B\sin\theta)+2rR_B^3\cos^2\theta\sin\theta -R_B\cos\theta\sqrt{M_B}}{2r R_B^2\cos^2\theta} \right)   \\
%%%%%%%%%%%%%%
\nonumber &= &  \dfrac{2}{\cos\theta}\ln\left( R_B\right)- \dfrac{2}{\cos\theta}\ln\left(\dfrac{R_B^2+r^2}{2r}-\dfrac{1}{2rR_B\cos\theta}\sqrt{N_B}\right)   \\ 
%%%%%%%%%%%%%%%%%%%%%%
\nonumber &=&  \dfrac{2}{\cos\theta}\ln\left( R_B\right)- \dfrac{2}{\cos\theta}\ln\left(\dfrac{R_B^2+r^2}{2r}-\dfrac{1}{2rR_B\cos\theta}\sqrt{R_B^6\cos^2\theta+r^4R_B^2\cos^2\theta-2r^2R_B^4\cos^2\theta}\right) \\
%%%%%%%%%%%%%%%%%%%%%%
\end{eqnarray}
\begin{eqnarray}
\nonumber &=&   \dfrac{2}{\cos\theta}\ln\left( R_B\right)- \dfrac{2}{\cos\theta}\ln\left(\dfrac{R_B^2+r^2}{2r}-\dfrac{1}{2r}\sqrt{(R_B^2-r^2)^2}\right)  \\
%%%%%%%%%%%%%%%%%%%%%%
\label{eqbeta}&=&    \dfrac{2}{\cos\theta}\ln\left(\dfrac{R_B}{r}\right).  
\end{eqnarray}
where $M_B:=R_B^2(R_B^2+ r^2-2rR_B\sin\theta)^2-(2rR_B^2\cos^2\theta-R_B\sin\theta(R_B^2+ r^2-2rR_B\sin\theta))^2$ and \\
$N_B:=-(4r^2R_B^4+R_B^6\sin^2\theta+R_B^2r^4\sin^2\theta-4rR_B^5\sin\theta-4r^3R_B^3\sin\theta+2R_B^4r^2\sin^2\theta)+R_B^6+r^4R_B^2+2r^2R_B^4+4r^2R_B^4\sin^2\theta-4rR_B^5\sin\theta-4r^3R_B^3\sin\theta$

\vspace{.5cm}

The area of the domain  $\D$ is given by
{\footnotesize
\begin{eqnarray}
\nonumber \area( \D)&=& \displaystyle\lim_{a\to 0^+} \left(\lim_{b\to+\infty}\int_{a}^{b}\int_{-y\tan\theta}^{d_3-y\tan\theta}\dfrac{dxdy}{y^2}-2\int_{a}^{R_B(1+\sin\theta)}\int_{\mu}^{\Gamma_B}\dfrac{dx dy}{y^2}-2\int_{a}^{R_A(1-\sin\theta)}\int_{\frac{d_3+2}{2}}^{\Gamma_A}\dfrac{dx dy}{y^2}\right)\\[10pt]
\nonumber&=& \displaystyle\lim_{a\to 0^+} \left(\dfrac{d_3}{a}-2\int_{a}^{R_B(1+\sin\theta)}\dfrac{\sqrt{R_B^2-(y-R_B\sin\theta)^2}}{y^2} dy-2\int_{a}^{R_A(1-\sin\theta)}\dfrac{\sqrt{R_A^2-(y+R_A\sin\theta)^2}}{y^2}\right)
\end{eqnarray}}
 where $\Gamma_B=\mu+\sqrt{R_B^2-(y-R_B\sin\theta)^2}$ and $\Gamma_A=\dfrac{d_3+2\mu}{2}+\sqrt{R_A^2-(y+R_A\sin\theta)^2}$. 
 
 \vspace{.5cm}
 
  Since 
 
 \vspace{.5cm}

 \begin{eqnarray*}
& &\int \dfrac{\sqrt{R_A^2-(y+R_A\sin\theta)^2}}{y^2}dy=-\arctan\left( \dfrac{y+R_A \sin\theta}{\sqrt{R_A^2-(y+R_A\sin\theta)^2}}\right)-\dfrac{\sqrt{R_A^2-(y+R_A\sin\theta)^2}}{y}+\\[10pt]
%%%%%%%%%%%%%%%%%%%%%%%%%%5
& & \hspace{2.5cm}+\tan\theta\left(-\ln y+\ln\left(2R_A^2\cos^2\theta-2R_A y\sin\theta+2R_A\cos\theta\sqrt{R_A^2-(y+R_A\sin\theta)^2}\right)\right)
\end{eqnarray*}
\vspace{.3cm}
and
\vspace{.3cm}
\begin{eqnarray*}
& &\int \dfrac{\sqrt{R_B^2-(y-R_B\sin\theta)^2}}{y^2}dy=\arctan\left( \dfrac{R_B \sin\theta-y}{\sqrt{R_B^2-(y-R_B\sin\theta)^2}}\right)-\dfrac{\sqrt{R_B^2-(y-R_B\sin\theta)^2}}{y}+\\[10pt]
%%%%%%%%%%%%%%%%%%%%%%%%%%5
& & \hspace{2.5cm}+\tan\theta\left(\ln y-\ln\left(2R_B^2\cos^2\theta+2R_B y\sin\theta+2R_B\cos\theta\sqrt{R_B^2-(y-R_B\sin\theta)^2}\right)\right),
\end{eqnarray*}
\vspace{.3cm}
we have
 \vspace{.3cm}
{\footnotesize \begin{eqnarray}
\nonumber \area( \D)&=& \displaystyle\lim_{a\to 0^+} \left( \dfrac{d_3}{a}-\dfrac{2(R_A\cos\theta+R_B \cos\theta)}{a}  \right)+ 2\pi+2\tan\theta\left(\ln\left(\dfrac{R_B(1+\sin\theta)}{R_A(1-\sin\theta)}\right)  + \ln\left(\dfrac{1-\sin\theta}{1+\sin\theta}\right)\right)\\[10pt]
%%%%%%%%%%%%%%%%%%%%%%%%%
\nonumber &=& 2\pi-2\tan\theta\,\ln\left(\dfrac{R_B}{R_A}\right),
\end{eqnarray} }
 with  $d_3=2(R_A\cos\theta+R_B \cos\theta)$. So the area of $ \D$ is given by
  \begin{equation} \label{aread}
  \area(\D)= 2\pi-2\tan\theta \ \ln\left(\dfrac{d_3-2\mu}{2\mu}\right).
  \end{equation}
 
 With these computations we arrive at the following Lemma
 
\begin{lemma} \label{lemma1}Let $\mu>0,  \, 0<H<1/2$, $d_3>\dfrac{2\mu }{1-4H^2}$  and $0<r<\dfrac{\mu}{2}$. Then  the function $G$ defined by
\begin{equation} 
G(\mu, d_3,H)= \alpha(\partial \D)-\beta(\partial \D)-2H\area(\D),
\end{equation}
is given by
\begin{equation}\label{a7}
G=\cos\theta\,\ln\left(\dfrac{d_3-2\mu }{2\mu} \right) -2\pi \sin\theta, 
\end{equation}
where $\theta=\arcsin(2H), \ 0<\theta<\pi/2$.
\end{lemma}

%%%%%%%%%%%
\begin{claim} \label{c1} Let $\mu>0,  0<H<1/2$ and  $\arcsin(2H)=\theta, \theta\in[0,\pi/2]$. Then
\begin{equation} \label{b1}
\cos^2\theta(1+e^{2\pi\tan\theta})>1.
\end{equation}
\end{claim}

\begin{proof}
Consider the real function $\xi(H)=(1-4H^2)\left(1+e^{\frac{4\pi H}{\sqrt{1-4H^2}}}\right)$, for $H\in(0,1/2)$. We have,  
$$
\displaystyle\lim_{H\to 0}\xi(H)=2 \hspace{.5cm} \textnormal{and}\hspace{.5cm}\dfrac{d \xi}{d H}(H)=e^{\frac{4\pi H}{\sqrt{1-4H^2}}}\left(\dfrac{4\pi}{\sqrt{1-4H^2}}-8H\right)-8H>0,
$$
so $\xi(H)>2$ for all $H\in(0,1/2)$, which proves the claim.  
\end{proof}

%%%%%%%%%%

%The existence of an admissible ideal domain is proved in the next Proposition. 

Now we are ready to prove Theorem \ref{teoquadrilatero}.

%\begin{proposition}\label{quadrilateroadmissivel} There exists  a  quadrilateral $\D$ which is admissible for any $\mu>0$ and $H, 0<H<1/2$. \end{proposition}

\begin{proof}[Proof of Theorem \ref{teoquadrilatero}]
 Set $d_3^*=(1+e^{2\pi\tan\theta})2\mu$. On the one hand, by  \eqref{b1}, 
$d_3^*>\dfrac{2\mu }{1-4H^2}=\dfrac{2\mu }{\cos^2\theta}$. 
So, $(\mu, d_3^*, H)$ belongs to the domain of $G$ and, by equation \eqref{a7}, 
 $$G(\mu,(1+e^{2\pi\tan\theta})2\mu, H)=0.$$

There are no admissible inscribed polygons (except the quadrilateral itself) in an admissible quadrilateral so the above quadrilateral satisfies the conditions of Theorem \ref{teoexist}.  This completes the proof of Theorem \ref{teoquadrilatero}.

\end{proof}

%%%%%%%%%%%%%%%%%%%%%%%%%%%%%%
%%%%%%%%%%%%%%%%%%%%%%%%%%%%%%
%%%%%%%%%%%%%%%%%%%%%%%%%%%%%%

\section{Extension of a Domain}\label{extension}

Let $\D=a_0a_1a_2\cdots a_{2n}$, with $a_{2n}=a_0$,  be an ideal admissible domain and $H, 0<H<1/2$. The sides $A_i$  are the arcs in $\partial\D$ joining $a_{2i}, a_{2i+1}$ and the arcs $B_i$ are the equidistant curves in $\partial\D$ joining $a_{2i+1}a_{2i+2}$ they satisfy $\kappa(A_i)=2H$ and $\kappa(B_i)=-2H$ with respect to $\D$. Also let $E$ and $E^\prime$ be two ideal admissible domains whose boundaries are ideal curved quadrilaterals,  $E=a_0 b_1 b_2 a_1$ and $E^\prime=a_1 b_3 b_4 a_2 $. The arcs $[a_0b_1] $ and $[a_1b_2]$ on the boundary of $E$ are convex with respect to $E$  the arcs $[a_0a_1]$ and $[b_1b_2]$ are concave with respect to $E$.  The arcs $[a_1a_2]$, $[b_3b_4]$ are convex and $[a_1b_3],\ [a_2b_4]$ are concave with respect to $E^\prime$.. The existence of such domains $E$ and $E^\prime$ was discussed in Section \ref{example}.    We use the following association between the vertices of $E, E\prime$ and the quadrilateral of Section \ref{example}:   $P_1 \leftrightarrow a_0, P_2\leftrightarrow a_1, P_3\leftrightarrow b_2$ and $P_4\leftrightarrow b_1$ to define the domain $E$ and $P_1 \leftrightarrow a_2, P_2\leftrightarrow b_4, P_3\leftrightarrow b_3$ and $P_4\leftrightarrow a_1$ in order to define the domain $E^\prime$.  Once $E$ and $E^\prime$ satisfy condition \eqref{cond1} we have 
\begin{equation}\label{condE}
\absc{a_0 b_1}+\absc{a_1 b_2}=\absc{a_0 a_1} +\absc{b_1 b_2}+2H\area(E),
\end{equation} 
\begin{equation}\label{condEprime}
\absc{a_1 a_2}+\absc{b_3 b_4}=\absc{a_1 b_3} +\absc{a_2 b_4}+2H\area(E^\prime),
\end{equation} 

\vspace{1cm}
To the domain $\D$ we attach the curved quadrilaterals $E$ and $E^\prime$ constructing a new domain $\D_0=\D\cup E\cup E^\prime= a_0 b_1 b_2 a_1 b_3 b_4 a_2 a_3\cdots a_{2n}$. See the the picture at the left in Figure \ref{fig2}.

\begin{figure}[h!]
\includegraphics[width=3.0in]{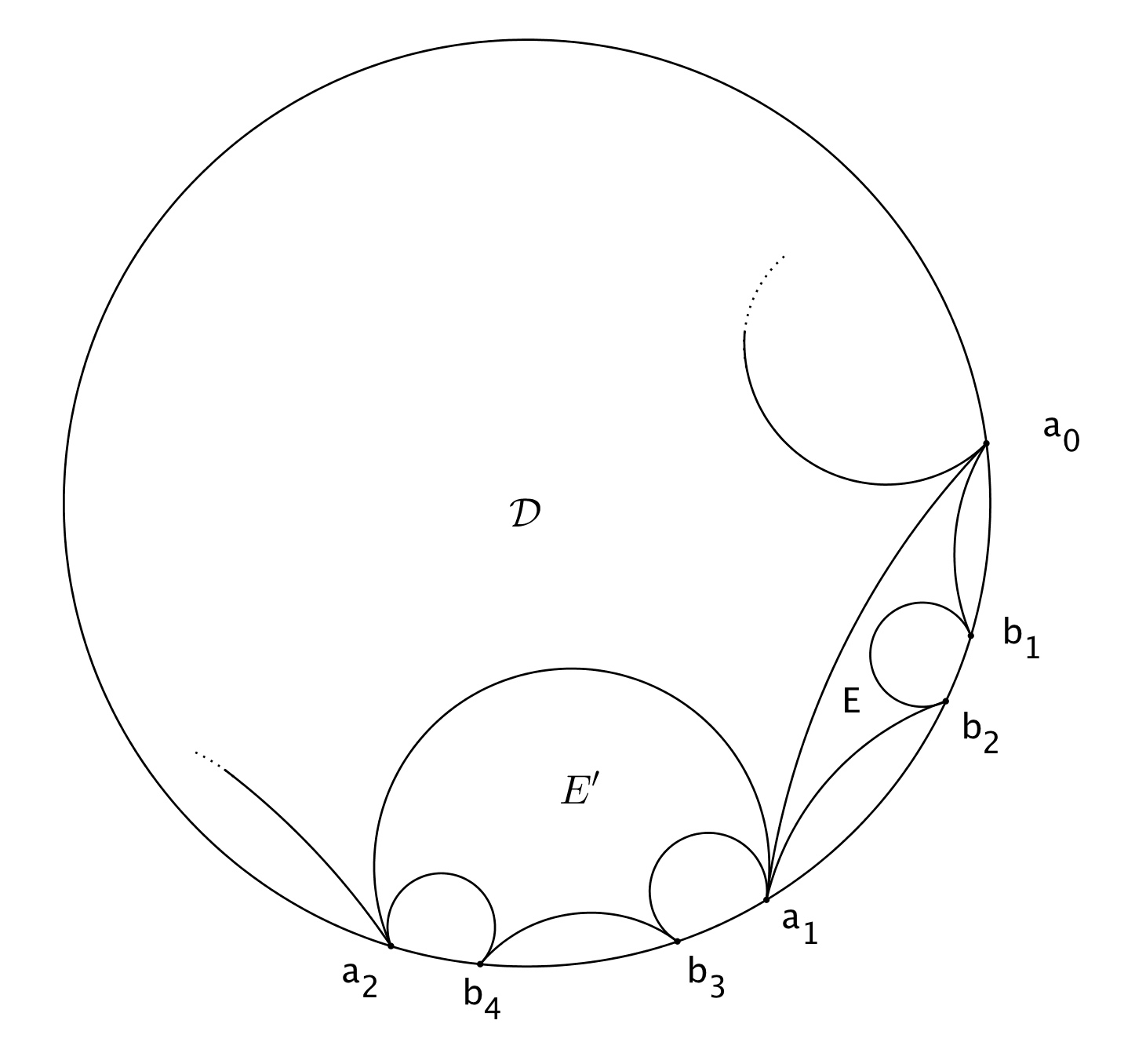}
\includegraphics[width=3.0in]{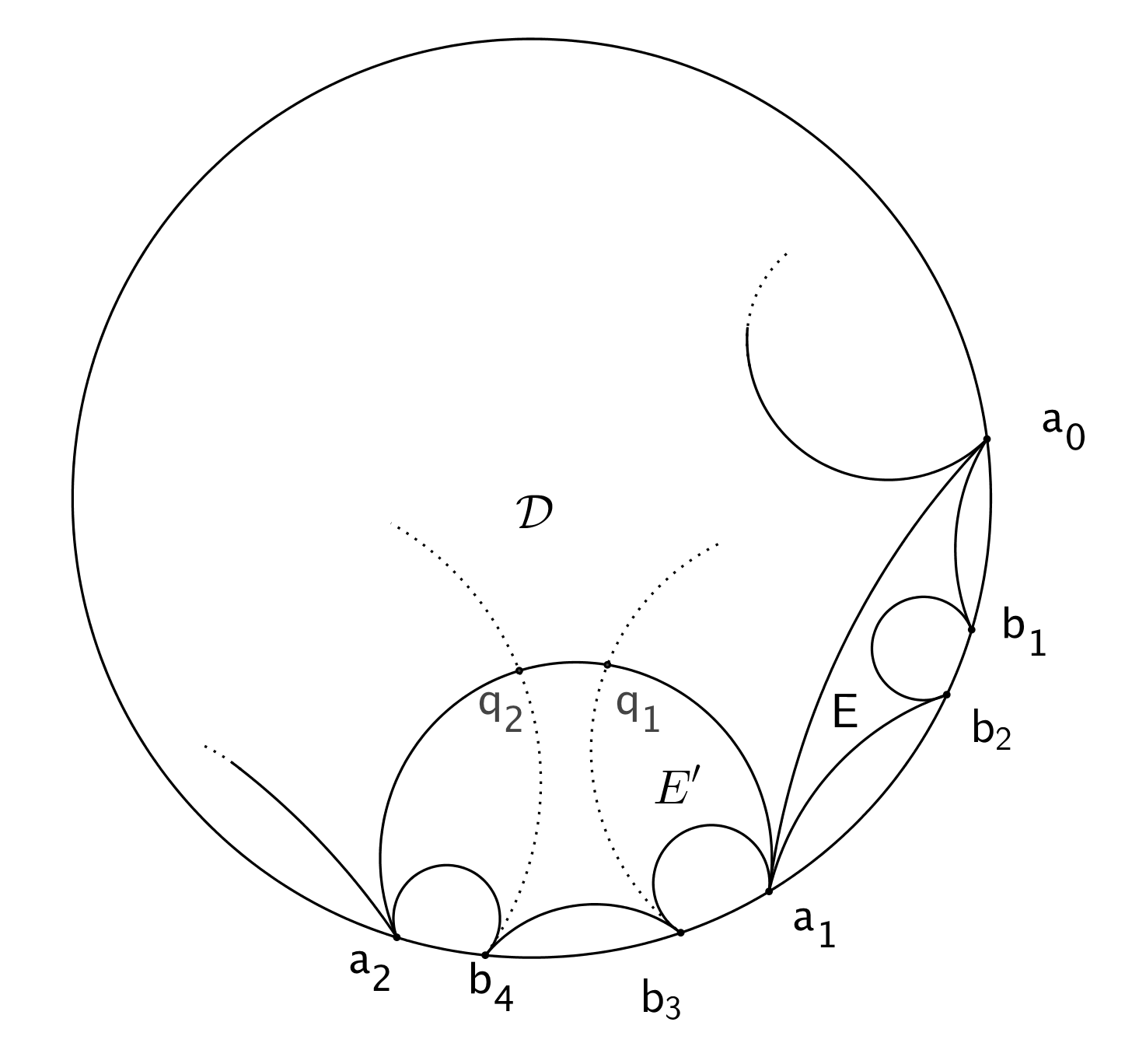}
\caption{}%{Figure 2}
\label{fig2}
\end{figure}

\begin{lemma}\label{lemma5} The domain $\D_0$ satisfies \eqref{cond2} for inscribed curved polygons $\pol$ in $\D_0$, with $\pol$ different from $\partial \D, \partial E,\partial  E^\prime,\partial ( \D\cup E), \partial (\D\cup E^\prime).$
\end{lemma}

\begin{remark} \label{obs1}  We observe that  $l(\pol)-2\alpha(\pol)$ is nondecreasing when we take a sequence of nested horocycles at the vertices of $\pol$. And it is increasing if we take a sequence of nested horocycles which are at a vertex of $\pol$ without a side $A_i$. A similar behaviour occurs for $l(\pol)-2\beta(\pol)$.  Since we want to prove the inequalities  $l(\pol)-2\alpha(\pol)>-2H\area(\poldo)$ and $l(\pol)-2\beta(\pol)>2H\area(\poldo)$, in the following we will just consider polygons having alternate sides $A_i$ and $B_j$. 
\end{remark}

\begin{proof}[Proof of Lemma \ref{lemma5}]
Let $\pol$ be an inscribed curved polygon in $\D_0$ with $\pol$ different from $ \partial (\D\cup E),$ $\partial \D,$ $ \partial E,$ $\partial  E^\prime,$ $  \partial (\D\cup E^\prime)$. We define $\polprime:=\pol\setminus E^\prime$. We claim the following.

\begin{claim} \label{claim4}
If $\abs{\polprime}-2\alpha_{\D_0}(\polprime)>-2H\area(\polprimedo)$, then 
$\abs{\pol} -2\alphadzero(\pol)>-2H\area(\poldo)$,
where $\poldo$ is the domain bounded by $\pol$, $\polprimedo$ is the domain bounded by $\polprime$ and $\pol$ is  different from $\partial\D, \partial\D_0,\partial( \D_0\setminus E), \partial(\D_0\setminus E^\prime)$ ..
\end{claim}

\begin{proof} [Proof of Claim \ref{claim4}] If $\polprime=\pol$ the claim is true. So, let us assume $\polprime\neq \pol$. By Remark \ref{obs1}, we can assume the arc $[b_3b_4]$ is contained in $\pol$. Let $d_1, d_2$ be two vertices of $\pol$ and  $[d_1 b_3]$ be the arc in $\pol$ joining $d_1$ and $b_3$, similarly let $[d_2b_4]$ be the  arc in $\pol$ joining $d_2$ and $b_4$. Observe that $d_1$  may equal $a_1$ and $d_2$ may equal $a_2$. We set $q_1:=[a_1a_2]\cap[d_1b_3]$ and $q_2:=[a_1a_2]\cap[d_2b_4]$.  Note that if $d_1=a_1$, then $q_1=a_1$ and if $d_2=a_2$, then $q_2=a_2$. 

See  the right  picture of Figure \ref{fig2}.

Then $\polprime$ and $\pol$ are related by

$
\begin{array}{rcl}
\alphadzero(\pol)&=&\alphadzero(\polprime)+\abs{[b_3b_4]}\\
\abs{\pol}&=&\abs{\polprime}-\abs{[q_1q_2]}+\abs{[q_2b_4]}+\abs{[q_1b_3]}+\abs{[b_3b_4]}.
\end{array}
$

Using the hypothesis, 
\begin{equation} \label{eqa1}
\abs{\pol}-2\alphadzero(\pol)>-2H\area(\polprimedo)-\absc{b_3b_4}-\absc{q_1q_2}+\absc{q_2b_4}+\absc{q_1b_3}.
\end{equation}
 $E^\prime$ is an ideal admissible domain, so let $u_{E^\prime}$ be a solution of the Dirichlet problem in $E^\prime$, the flux of $u_{E^\prime}$ on the curved quadrilateral $q_1q_2 b_3 b_4$ gives
 
 \begin{figure}[h!]
\centering
\includegraphics[width=3.0in]{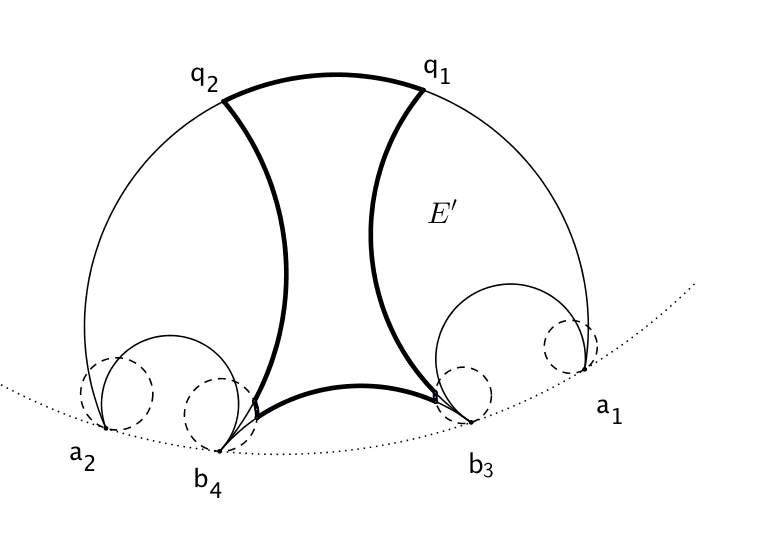}
\caption{}
\label{fig40}
\end{figure}

\begin{equation}\label{eqa2}
2H\area((q_1q_2 b_3 b_4)^T)=\absc{q_1q_2}+\absc{b_3b_4}+F_{u_{e^\prime}}([q_2b_4])+F_{u_{E^\prime}}([q_1b_3])+F_{u_{E^\prime}}(\gamma),
\end{equation}
where $(q_1q_2 b_3 b_4)^T$ is the domain $q_1q_2 b_3 b_4$ truncated by the horocycles $\{\hor_i\}$;  $\gamma$ is the intersection of the horocycles with $q_1q_2 b_3 b_4$. Since $u_{E^\prime}$ is continuous on $[q_2b_4]$ and $[q_1b_3]$, there exists a constant $c>0$ such that $F_{u_{e^\prime}}([q_2b_4])+F_{u_{E^\prime}}([q_1b_3])>-\absc{q_2b_4}-\absc{q_1b_3} +c$. On the other hand $F_{u_{e^\prime}}(\gamma)$ and $2H\area(q_1q_2 b_3 b_4)-2H\area((q_1q_2 b_3 b_4)^T)$ can be made as small as we want by choosing horocycles "small" enough. So, by \eqref{eqa2} 
     \begin{eqnarray}
     \nonumber 
     2H\area(q_1q_2b_3b_4)&>&\absc{q_1q_2}+\absc{b_3b_4}-\absc{q_2b_4}-\absc{q_1b_3}+c+F_{u_{E^\prime}}(\gamma)+2H\area(q_1q_2b_3b_4\setminus(q_1q_2b_3b_4)^T)\\
    \label{eqa3} &>&\absc{q_1q_2}+\absc{b_3b_4}-\absc{q_2b_4}-\absc{q_1b_3}.
     \end{eqnarray}
Equations \eqref{eqa1} and \eqref{eqa3} give, 
\begin{equation} \label{eqa4}
\abs{\pol}-2\alphadzero>-2H\area(\polprimedo)-2H\area(q_1q_2b_3b_4)=-2H\area(\poldo).
\end{equation}

\end{proof}

We now define  $\widetilde{\pol}:=\pol\setminus E$.   

\begin{claim}\label{claim5}
If $\abs{\widetilde{\pol}}-2\betadzero(\widetilde{\pol} )>2H\area(\widetilde{\poldo}) $, then $\abs{\pol}-2\betadzero(\pol )>2H\area(\poldo) ,$ where $\poltiodo$ is the domain bounded by $\poltio$, $\poldo$ is the domain bounded by $\pol$ and $\pol$ is  different from $\partial\D, \partial\D_0,\partial( \D_0\setminus E), \partial(\D_0\setminus E^\prime)$.

\end{claim}

\begin{proof}[Proof of Claim \ref{claim5}]   If $\widetilde{\pol}=\pol$ the claim holds.  Let us assume that $\widetilde{\pol}\neq \pol$, taking into account the Remark \ref{obs1}, we can assume $[b_1b_2]\subset\pol$. Let $[d_1b_1], \ [d_2b_2]$ be arcs in $\pol$ with $d_1\neq b_1, b_2$, $d_2\neq b_1, b_2$ and $d_1\neq d_2$. Denote $q_1=[a_0a_1]\cap[b_1d_1]$ and $[q_2]=[a_0a_1]\cap[b_2d_2]$.  We have, 

$
\begin{array}{rcl}
\abs{\pol}&=&\abs{\widetilde{\pol}}-\absc{q_1q_2}+\absc{q_1b_1}+\absc{q_2b_2}+\absc{b_1b_2}\\
\betadzero(\pol)&=&\betadzero(\poltio)+\absc{b_1b_2}.
\end{array}
$

Then, by hypothesis
\begin{equation}\label{eqa5}
\abs{\pol}-2\betadzero(\pol)>2H\area(\poltiodo)-\absc{q_1q_2}-\absc{b_1b_2}+\absc{q_1b_1}+\absc{q_2b_2}. 
\end{equation}

Let $u_E$ be a solution of the Dirichlet problem in $E$. The flux of $u_E$ on the curved quadrilateral $q_1q_2b_1b_2$ gives 

\begin{equation}\label{eqa6}
2H\area((q_1q_2b_2b_1)^T)=-\absc{q_1q_2}-\absc{b_1b_2}+F_{u_E}(q_1b_1)+F_{u_E}(q_2b_2) +F_{u_E}(\gamma),
\end{equation}
where $(q_1q_2 b_2 b_1)^T$ is the truncated curved quadrilateral, $\gamma$ is the intersection of the horocycles with $q_1q_2 b_2 b_1$. Since $u_E$ is continuous  on $[q_1b_1]$ and $[q_2b_2]$, there exists $c>0$, such that  $F_{u_E}(q_1b_1)+F_{u_E}(q_2b_2)<\absc{q_1b_1}+\absc{q_2b_2}-c$. Then, by equation \eqref{eqa6}, 
{\footnotesize \begin{eqnarray}
\nonumber 2H\area(q_1q_2b_2b_1) &<&-\absc{q_1q_2} -\absc{b_2b_1}+\absc{q_1b_1}+\absc{q_2b_2}-c+F_{u_E}(\gamma)+2H\area(q_1q_2b_2b_1\setminus(q_1q_2b_2b_1)^T)\\
\label{eqa7}&<&-\absc{q_1q_2} -\absc{b_2b_1}+\absc{q_1b_1}+\absc{q_2b_2},
\end{eqnarray}}
the last inequality holds since $F_{u_E}(\gamma)$ and $2H\area(q_1q_2b_2b_1\setminus(q_1q_2b_2b_1)^T)$ are arbitrarily small for a choice of horocycles small enough.  By \eqref{eqa5} and \eqref{eqa7}, we obtain 

\begin{equation}
\label{eqa8}
\abs{\pol}-2\betadzero(\pol)>2H\area(\poldo),
\end{equation} 
which proves the claim.
\end{proof}

So, in order to prove Lemma \ref{lemma5}, we need to show that 
\begin{equation}\label{eqa9}
\abs{\polprime}-2\alphadzero(\polprime)>-2H\area(\polprimedo)\ \ \ \textnormal{and}\ \ \ \abs{\poltio}-2\betadzero(\poltio)>2H\area(\poltiodo),
\end{equation}
for all inscribed polygons different from $\partial\D, \partial\D_0,\partial( \D_0\setminus E), \partial(\D_0\setminus E^\prime)$. 

We start proving the first inequality in \eqref{eqa9}. We define $\pol^{\prime\prime}:=\polprime\setminus E$,  $\pol^{\prime\prime}$ is contained in $\D$. We write $\pol^{\prime\prime}=I_0\cup I_1\cup J$, where $I_0$ are all arcs $A_i's$ on $\pol^{\prime\prime}$ which are not contained in $[a_0a_1]$; $I_1:= [a_0a_1]\cap \pol^{\prime\prime}$; $J:=\pol^{\prime\prime}\setminus(I_1\cup I_0)$. 

Let $u_\D$ be a solution of the Dirichlet problem  in $\D$. The flux of $u_\D$ on $\pol^{\prime\prime}$ gives
\begin{equation}\label{eqa10}
2H\area(\poltprimedo^T)=\alpha_{\D_0}(\poltprime)+\abs{I_1}+F_{u_\D}(J)+F_{u_\D}(\gamma),
\end{equation}
where $\gamma$ are the arcs of horocycles  inside $\poltprimedo$, and $\poltprimedo^T$ is the truncated domain $\poltprimedo$. 

On the other hand, $\abs{\poltprime}=\alpha_{\D_0}(\poltprime)+\abs{I_1}+\abs{J}$. So, 
by \eqref{eqa10}, 
\begin{eqnarray}
\nonumber \abs{\poltprime}-2\alpha_{\D_0}(\poltprime)&=&-2H\area(\poltprimedo)+2H\area(\poltprimedo\setminus \poltprimedo^T)+2\abs{I_1}+F_{u_\D}(J)+F_{u_\D}(\gamma)+\abs{J}\\
\label{eqa11} &>& -2H\area(\poltprimedo)+2\abs{I_1}, 
\end{eqnarray}
the last inequality follows from the fact that  $\area(\poltprimedo\setminus \poltprimedo^T)$ and $F_{u_\D}(\gamma)$ tend to zero for a sequence of nested horocycles and since $\pol$ is different from $\D$,  there is a constant $c>0$, such that $F_{u_\D}(J)>-\abs{J}+c$.

We need to consider some cases.

{\bf \emph{Case 1.} } $[a_0b_1]$ and $[a_1 b_2]$ are in $\pol$. 

We have, 

 $\begin{array}{rcl}
 \alpha_{\D_0}(\polprime)&=&2\alpha_{\D_0}(\poltprime)+\absc{a_0b_1}+\absc{a_1b_2},\\[7pt]
\abs{\polprime}&=&\abs{\poltprime}-\absc{a_0a_1}+\absc{a_0b_1}+\absc{b_1b_2}+\absc{b_2a_1},\\ [7pt]
\abs{I_1}&=&\absc{a_0a_1}.
 \end{array}
 $

Then, by inequality \eqref{eqa11} and \eqref{condE}
\begin{eqnarray*}
\abs{\polprime}-2\alpha_{\D_0}(\polprime)&=&\abs{\poltprime}-\alpha_{\D_0}(\poltprime)-\absc{a_0b_1}-\absc{a_1b_2}-\absc{a_0a_1}+\absc{b_1b_2}\\[5pt]
&>&-2H\area(\poltprimedo)+\absc{a_0a_1}+ \absc{b_1b_2}-\absc{a_0b_1}-\absc{a_1b_2}\\[5pt]
&=& -2H\area(\poltprimedo)-2H\area(E)\\[5pt]
&=& -2H\area(\polprimedo).
\end{eqnarray*}

{\bf \emph{Case 2.}} $[a_0b_1]$ is contained in $\polprime$ and $[a_1b_2]$ is not on $\polprime$. 

By remark \ref{obs1}, we can assume $[b_1b_2]$ is not on $\polprime$. We have, 

$
\begin{array}{rcl}
\alpha_{\D_0}(\polprime)&=&\alpha_{\D_0}(\poltprime)+\absc{a_0b_1},\\[7pt]
\abs{\polprime}&=&\abs{\poltprime}-\abs{I_1}+\absc{a_0b_1}+\absc{qb_1}\\[5pt]
\abs{I_1}&=&\absc{a_0q}, 
\end{array}
$
\vspace{.4cm}\newline
where $q=(\pol\cap[a_0a_1])\setminus\{a_0\}$. Then, by \eqref{eqa11}
\begin{equation}\label{eqa12}
\abs{\polprime}-2\alphadzero(\polprime)>-2H\area(\poltprimedo)+\abs{I_1}-\absc{a_0b_1}+\absc{qb_1}.
\end{equation}
The flux of a solution $u_E$ of the Dirichlet problem  in $E$, gives
\begin{eqnarray*}
2H\area((\polprimedo\cap E)^T)&=&\absc{a_0b_1}-\absc{a_0q}+F_{u_E}([b_1q])+F_{u_E}(\gamma)\\[5pt]
&>& \absc{a_0b_1}-\absc{a_0q}-\absc{b_1q}+c, 
\end{eqnarray*}
for some $c>0$, since $u_E$ is continuous  on $[b_1q]$. The area $\area((\polprime\cap E)\setminus (\polprime\cap E)^T)$ and $\abs{\gamma}$ tend to zero for a sequence of nested horocycles, so
\begin{equation}\label{eqa13}
\abs{I_1}-\absc{a_0b_1}+\absc{qb_1}>-2H\area(\polprimedo\cap E).
\end{equation}
By \eqref{eqa12} and \eqref{eqa13}, we have
\begin{equation}\label{eqa14}
\abs{\polprime}-2\alphadzero(\polprime)>-2H\area(\polprimedo).
\end{equation}

{\bf \emph{Case 3.}} $[a_1b_2]$ is contained in $\polprime$ and $[a_0b_1]$ is not on $\polprime$. 
 This case is similar to case 2.
 
{\bf \emph{Case 4.}} $\poltprime=\polprime$.  

This case follows directly from inequality \eqref{eqa11}.
 
 These are the cases to be considered in order to prove the first inequality of \eqref{eqa9}.  Now, let us  prove the second inequality in \eqref{eqa9}.  We define $\polttio=\poltio\setminus E^\prime$. We write $\polttio=J_0\cup J_1\cup L$, where 
 $J_0$ are all arcs $B_i$ on $\polttio$ which are not contained in $[a_1a_2]$;\  $J_1=\polttio\cap[a_1a_2]$ and $L=\polttio\setminus(J_0\cup J_1)$. We denote $\polttiodo$ the domain bounded by $\polttio$ and by $\polttiodot$ the domain $\polttiodo$ truncated by the horocycles at the vertices of $\pol$. 
 
Since $\D$  is an ideal admissible domain, let $u_\D$ be a solution of the Dirichlet problem in $\D$. The flux of $u_\D$ in $\polttio$

\begin{eqnarray}
\nonumber 
2H\area(\polttiodot)&=&F_{u_\D}(J_0)+F_{u_\D}(J_1)+F_{u_\D}(L)+F_{u_\D}(\gamma)\\[5pt]
\label{eqa15}&=&-\betadzero(\polttio)-\abs{J_1}+F_{u_\D}(L)+F_{u_\D}(\gamma),
\end{eqnarray} 
 where $\gamma$ are the arcs of horocycles in $\polttio$.  On the other hand, 
 \begin{equation}
\label{eqa16} \abs{\polttio}-\betadzero(\polttio)=\abs{J_1}+\abs{L}. 
\end{equation}  
Moreover, since $u_\D$ is continuous on $L$ and $\gamma$, there exists a $c>0$ such that,
\[
F_{u_\D}(L)<\abs{L}-c.
\]  
Since $\abs{\gamma}$  and $\area(\polttiodo\setminus \polttiodot)$ tend to zero for a sequence of nested horocycles at the vertices of $\pol$, by \eqref{eqa15} and \eqref{eqa16} we obtain
 \begin{equation}
 \label{eqa17}
\abs{\polttio}-2\betadzero(\polttio)>2H\area(\polttiodo)+2\abs{J_1}. 
 \end{equation}
 
 We have some cases to consider.
 
 {\bf \emph{Case 1.}} $[a_1b_3]$ and $[a_2b_4]$ are in $\pol$ (so, taking into account Remark \ref{obs1},  $[b_3b_4]\subset\pol$).
 
 We have,
 \begin{eqnarray*}
 \abs{\poltio}&=&\abs{\polttio}-\abs{J_1}+\abs{[a_1b_3]}+\abs{[b_3b_4]}+\abs{[b_4a_2]}\\[10pt] 
 \betadzero(\poltio)&=& \betadzero(\polttio)+\abs{[a_1b_3]} +\abs{[b_4a_2]}\\[10pt] 
 J_1&=&[a_1a_2]\\[10pt]
 \Longrightarrow \abs{\poltio}-2\betadzero(\poltio)&=& \abs{\polttio}-2\betadzero(\polttio)-\abs{J_1}-\abs{[a_1b_3]}+\absc{b_3b_4}-\absc{b_4a_2}.
 \end{eqnarray*}
 By \eqref{eqa17}, 
 \begin{equation}
 \label{eqa18} \abs{\poltio}-2\betadzero(\poltio)>2H\area(\polttiodo)+\absc{a_1a_2}-\absc{a_1b_3}-\absc{b_4a_2}+\absc{b_3b_4}.
 \end{equation}
 Using \eqref{condEprime} and  \eqref{eqa18} we obtain
$$ \abs{\poltio}-2\betadzero(\poltio)>2H\area(\polttiodo)+2H\area(E^\prime)= 2H\area(\poltiodo), $$
 as desired.
 
 {\bf \emph{Case 2.}} $[a_1b_3]$ is contained in $\poltio$ and $[a_2b_4]$ is not in $\poltio$. 
  
  We have, 
  \begin{eqnarray*}
  \abs{\poltio} &=& \abs{\polttio}-\abs{J_1}+\absc{a_1b_3}+\absc{q_1b_3}\\[10pt]
  \betadzero(\poltio)&=& \betadzero(\polttio)+\absc{a_1b_3}\\[10pt]
  J_1&=&a_1q_1,
\end{eqnarray*}   
 where $q_1\in\polttio$ is such that $J_1=a_1q_1$.  So, by \eqref{eqa17}
 \begin{eqnarray}
\nonumber \abs{\poltio}-2\betadzero(\poltio)&=&\abs{\polttio} -\abs{J_1}-\absc{a_1b_3}+\absc{q_1b_3}-2\betadzero(\polttio)\\[10pt]
 \label{eqa22}
 &>&2H\area(\polttiodo)+\absc{a_1q_1} -\absc{a_1b_3}+\absc{q_1b_3}.
 \end{eqnarray}
 
 Let  $u_{E^\prime}$ be a solution of the Dirichlet problem in $E^\prime$. The flux of  $u_{E^\prime}$ on the curved triangle  $a_1q_1b_3$  is

\[
2H\area((a_1q_1b_3)^T)=\absc{a_1q_1}-\absc{a_1b_3}+F_{u_{E^\prime}}(q_1b_3)+F_{u_{E^\prime}}(q_1b_3)(\gamma),
\]
 where $\gamma$ are the arcs of horocycles in $[a_1q_1b_3]$. We observe that $u_{E^\prime}$ has continuous boundary values on $[q_1b_3]$ and  $\area(a_1q_1b_3\setminus (a_1q_1b_3)^T)$ and $\abs{\gamma}$ tends to zero for a sequence of nested horoclycles at the vertices of $\pol$, hence 
 \begin{equation}
 \label{eqa20} 2H\area(a_1q_1b_3)<\absc{a_1q_1}-\absc{a_1b_3}+\absc{q_1b_3}.
\end{equation}   
By\eqref{eqa17} and \eqref{eqa10}, we obtain
   $$\abs{\poltio}-2\beta(\poltio)>2H\area(\poltiodo).$$
 
 { \bf \emph{Case 3.}} $[a_2b_4]$ is in $\poltio$ and $[a_1b_3]$ is not contained on $\poltio$.
 
 This case is similar to Case 2.    
 
  {\bf \emph{Case 4.}} $\poltio$  is contained in $\D$. 
  
  In this case, $\polttio=\poltio$ so inequality \eqref{eqa17} gives us the result.

This concludes the proof of Lemma \ref{lemma5}.
  \end{proof}

%%%%%%%%%%%%%%%%%%%%%%%%%%%%%%

Unfortunately, the domain $\D_0$ is not an ideal admissible domain. It is not possible to show conditions \eqref{cond1} and \eqref{cond2} for curved polygons which bound $E, E\prime$ and its complements. So, in order to proceed we will do a small perturbation of the vertices of $E$ and $E\prime$. 
  
Lemma \ref{lemma1}  implies that $G$ is strictly monotone with respect to the variable $d_3$, so  there exists a $\tau>0$ such that, see Figure \ref{fig3}
\begin{equation}
\label{eqa21}\left\{
\begin{array}{ccl}
\varphi(\tau)&=& -\absc{a_0b_1}-\absc{a_1b_2(\tau)}+\absc{b_1b_2(\tau)}+\absc{a_0a_1}+2H\area(E_\tau)>0\\
\varphi(\tau)&=& \absc{a_1a_2}+\absc{b_3(\tau)b_4}-\absc{a_1b_3(\tau)}-\absc{a_2b_4}-2H\area(E_\tau^\prime)>0
\end{array}
\right..
\end{equation}
\begin{figure}[h] 
\centering
\includegraphics[width=2.5in]{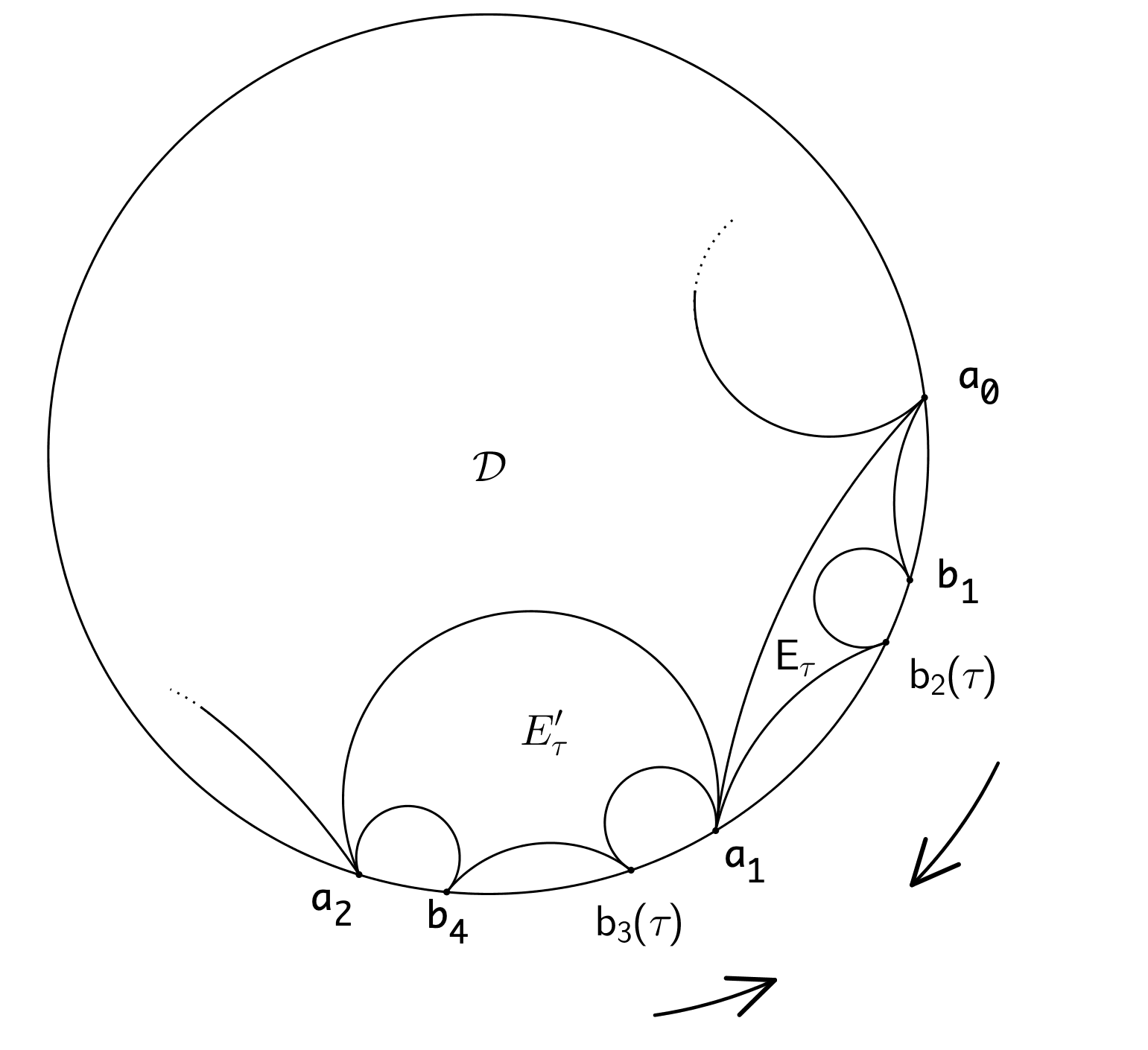}
\caption{Domain $\D_\tau$}
\label{fig3}
\end{figure}

Let $\dtau$ be the domain obtained attaching $E_\tau$ and $E_\tau^\prime$ to $\D$. We will show that there exist a solution of the Dirichlet problem in $\dtau$. First we analyse the condition \eqref{cond1}.  We have 

\begin{eqnarray*}
\alphadtau(\partial\dtau)&=&\alpha_\D(\partial\D)-\absc{a_0a_1}+\absc{a_0b_1}+\absc{b_2(\tau)a_1}+\absc{b_4b_3(\tau)}\\[5pt]
\betadtau(\partial\dtau)&=& \betad(\partial\D)-\absc{a_1a_2}+\absc{b_1b_2(\tau)}+\absc{a_1b_3(\tau)}+\absc{a_2b_4}.
\end{eqnarray*}
Since $\D$ is an ideal admissible domain, using \eqref{eqa21}, we obtain
\begin{eqnarray*}
\alphadtau(\partial\dtau)-\betadtau(\partial\dtau)&=&\alphad(\partial\D)-\betad(\partial\D)+2H\area(E_\tau)-\varphi(\tau)+2H\area(E^\prime_\tau)+\varphi(\tau)\\[5pt]
&=&2H\area(\dtau).
\end{eqnarray*}

\begin{lemma}\label{lemma10} Let $\dtau$ be the ideal domain defined above. There is a solution of the Dirichlet problem  in $\dtau$.  

\end{lemma}
\begin{proof}
We will prove the condition \eqref{cond2} of Theorem \ref{teoexist}. Let $\pol$ be a curved  inscribed polygon in $\D_\tau$. In the following,  \eqref{eqa21} will be used without mention.   

\begin{enumerate}
[i.\ ]
%%%%%%%%%%%%%%%%%%
\item Suppose $\pol=\partial E_\tau$. Then
\begin{eqnarray*}
\abs{\partial E_\tau}-2\alphadtau(\partial E_\tau)&=& -\absc{a_0b_1}-\absc{b_2(\tau)a_1}+\absc{b_1b_2(\tau)}+\absc{a_0a_1}\\[5pt]
&=& -2H\area(E_\tau)+\varphi(\tau)\\
&>& -2H\area(E_\tau). 
\end{eqnarray*}

\begin{eqnarray*}
\abs{\partial E_\tau}-2\betadtau(\partial E_\tau)&=& \absc{a_0b_1}+\absc{b_2(\tau)a_1}+\absc{a_0a_1}-\absc{b_1b_2(\tau)}\\
&=& 2H\area(E_\tau)+2\absc{a_0a_1}-\varphi(\tau)\\
&>& 2H\area(E_\tau).
\end{eqnarray*}
The last inequality holds since $\varphi(\tau) $ is small when compared with $\absc{a_0a_1}$.
\vspace{.7cm}
%%%%%%%%%%%%%%%%%%%%%%%%%%%%%%%
\item Suppose $\pol=\partial E^\prime_\tau$. Then

\begin{eqnarray*}
\abs{\partial E^\prime_\tau}-2\alphadtau(\partial E^\prime_\tau)&=& \absc{a_1a_2}+\absc{a_1b_3(\tau)}-\absc{b_3(\tau)b_4}+\absc{b_4a_2}\\[5pt]
&=& 2\absc{a_1a_2}-2H\area(E^\prime_\tau)-\varphi(\tau)\\[5pt]
&>&- 2H\area(E^\prime_\tau),
\end{eqnarray*}
the last inequality holds since $\varphi(\tau) $ is small when compared with $2\absc{a_1a_2}$. 

\begin{eqnarray*}
\abs{\partial E^\prime_\tau}-2\betadtau(\partial E^\prime_\tau)&=&\absc{a_1a_2}+\absc{b_4b_3(\tau)}-\absc{a_1b_3(\tau)}-\absc{a_2b_4}\\[5pt]
&=&2H\area(E^\prime_\tau)+\varphi(\tau)\\[5pt]
&>&2H\area(E^\prime_\tau).
\end{eqnarray*}
%%%%%%%%%%%%%%%%%%%%%%%%%%%%
\item Suppose $\pol=\partial\D=\partial \D_\tau\setminus(E_\tau\cup E_\tau^\prime)$.

Since $\D$ is an ideal admissible domain, we have

\begin{eqnarray*}
\abs{\pol}-2\alphadtau(\pol)&=&\abs{\partial \D}-2\alpha_\D(\partial \D)+2\absc{a_0a_1}\\[5pt]
&=&\beta_\D(\partial \D)-\alpha_\D(\partial \D)+2\absc{a_0a_1}\\[5pt]
&>& -2H\area(\D).
\end{eqnarray*}
\vspace{.1cm}
\begin{eqnarray*}
\abs{\pol}-2\betadtau(\pol)&=&\abs{\partial\D}-2\beta_\D(\partial\D)+2\absc{a_1a_2}\\[5pt]
&=&\alpha_\D(\partial\D)-\beta_\D(\partial\D)+2\absc{a_1a_2}\\[5pt]
&>&2H\area(\D).
\end{eqnarray*}

\vspace{.1cm}
%%%%%%%%%%%%%%%%%%%%%%
\item Suppose $\pol=\partial(\D_\tau\setminus E_\tau)=\partial(\D\cup E^\prime_\tau)$.

We use again that  $\D$ is an ideal admissible domain and that $\varphi(\tau) $ is small.
\begin{eqnarray*}
\abs{\pol}-2\alphadtau(\pol)&=&\abs{\partial\D} -2\alpha_\D(\partial\D)+2\absc{a_0a_1}-\absc{a_1a_2}+\absc{a_1b_3(\tau)} -\absc{b_3(\tau)b_4}+\absc{a_2b_4} \\[5pt]
&=&-2H\area(\D)-2H\area(E_\tau^\prime)-\varphi(\tau)+2\absc{a_0a_1}\\[5pt]
&> & -2H\area(\D\cup E^\prime_\tau)=-2H\area(\poldo).
\end{eqnarray*}
\vspace{.1cm}
\begin{eqnarray*}
\abs{\pol}-2\betadtau(\pol)&=& \abs{\partial\D}-2\beta_\D(\partial\D)+\absc{a_1a_2}-\absc{a_1b_3(\tau)}-\absc{a_2b_4} +\absc{b_3(\tau)b_4}\\[5pt]
&=& 2H\area(\D)+2H\area(E^\prime_\tau)+\varphi(\tau)\\[5pt]
&>&2H\area(\D\cup E^\prime_\tau)=2H\area(\poldo).
\end{eqnarray*}
\vspace{.1cm}
%%%%%%%%%%%%%%%%%%%%%%%%%%%%%%%%%%%%

\item Suppose $\pol=\partial(\D\setminus E^\prime_\tau)=\partial(\D\cup E_\tau)$. Then
\begin{eqnarray*}
\abs{\pol}-2\alphadtau(\pol)&=&\abs{\partial\D}-2\alpha_\D(\partial\D)-\absc{a_0b_1}+\absc{a_0a_1}-\absc{b_2(\tau)a_1}+\absc{b_1b_2(\tau)}\\[5pt]
&=& -2H\area(\D)-2H\area(E_\tau)+\varphi(\tau)\\[5pt]
&>&-2H\area(\D\cup E_\tau)=-2H\area(\poldo).
\end{eqnarray*}
\vspace{.1cm}
\begin{eqnarray*}
\abs{\pol}-2\betadzero(\pol)&=&\abs{\partial\D}-2\beta_\D(\partial\D)-\absc{a_0a_1}-\absc{b_1b_2(\tau)}+\absc{a_0b_1}+\absc{a_1 b_2(\tau)}+2\absc{a_1a_2}\\[5pt]
&=&2H\area(\D)+2H\area(E_\tau)-\varphi(\tau)+2\absc{a_1a_2}\\[5pt]
&>&2H\area(\D)+2H\area(E_\tau)=2H\area(\poldo).
\end{eqnarray*}
The last inequality holds since $\varphi(\tau)$ is small when compared with $\absc{a_1a_2}$.

\vspace{.1cm}
%%%%%%%%%%%%%%%%%%%%%%%%%%%%%%%%%%%%%
\item Suppose $\pol=E_\tau\cup E^\prime_\tau$. Then

\begin{eqnarray*}
\abs{\pol}-2\alphadtau(\pol)&=&\absc{a_0a_1}+\absc{b_1b_2(\tau)}-\absc{a_0b_1}-\absc{b_2(\tau)a_1}+\\[3pt]
& & + \absc{a_1a_2} +\absc{a_1b_3(\tau)} -  \absc{b_3(\tau)b_4} + \absc{a_2b_4}\\[5pt]
&=&-2H\area(E_\tau)+\varphi(\tau)-2H\area(E^\prime_\tau)+2\absc{a_1a_2}-\varphi(\tau)\\[5pt]
&>&-2H\area(E_\tau\cup E_\tau^\prime)=-2H\area(\poldo).
\end{eqnarray*}
\vspace{.1cm}
\begin{eqnarray*}
\abs{\pol}-2\betadtau(\pol)&=& \absc{a_0a_1}+\absc{a_0b_1}+\absc{b_2(\tau)a_1}-\absc{b_1b_2(\tau)}+\\[3pt]
& &+ \absc{a_1a_2}+\absc{b_4b_3(\tau)}-\absc{a_1b_3(\tau)}-\absc{b_4a_2}\\[5pt]
&=& 2H\area(E_\tau)+2\absc{a_0a_1}-\varphi(\tau)+2H\area(E_\tau^\prime)+\varphi(\tau)\\[5pt]
&>&2H\area(E_\tau\cup E_\tau^\prime)=2H\area(\poldo).
\end{eqnarray*}
\item Suppose $\pol$ is different from $\partial\dtau, \partial E, \partial E^\prime$ and their complements.

It was proved in Lemma \ref{lemma5} that inequalities of condition \eqref{cond2} holds for $\tau=0$. So, for a small $\tau$ these inequalities are preserved. So they hold for inscribed polygons $\pol$ in $\dtau$, different from $\partial \dtau, \partial E, \partial E^\prime, \partial(\dtau\cup E), \partial(\dtau\cup E^\prime), \partial(E\cup E^\prime), \partial \D$. 
\end{enumerate}

This  completes the proof of the lemma.
\end{proof}

We proved that $\D_\tau $ is an ideal admissible domain, so, there is a solution of the Dirichlet problem in $\dtau$ that we denote by $v_\tau$. We now derive several properties of $v_\tau$.

\begin{lemma}\label{lemaconvergencia}
Let $v_\tau$ be a solution of the Dirichlet problem  in $\D_\tau$ and $u$ be a solution of the Dirichlet problem  in $\D$. Then, 
\begin{equation}\label{eqa27}
\lim_{\tau\to 0}\nabla v_\tau\vert_\D=\nabla u,
\end{equation}
where $\nabla$ denotes the gradient. 
\end{lemma}
\begin{proof}
Let $X_\tau:=\dfrac{\nabla v_\tau}{W_\tau}, \  \ X:=\dfrac{\nabla u}{W}$, where $W_\tau:=\sqrt{1+\vert \nabla v_\tau\vert^2}$ and $W:=\sqrt{1+\vert \nabla u\vert^2}$\ . We prove that $\lim_{\tau\to 0} X_\tau\vert_\D=X$.

First, we show equality \eqref{eqa27} for points in  $\partial\D\setminus\left([a_0a_1]\cup [a_1a_2]\right)$. Let $\nu$ be the unit conormal pointing outside $\partial\D$. On $\partial\D\setminus\left([a_0a_1]\cup [a_1a_2] \right)$ both solutions $v_\tau$ and $u$ have the same data $+\infty$ or $-\infty$, then 
\[ 
X_\tau\vert_{\partial\D\setminus\left([a_0a_1]\cup[a_1a_2]\right)}=X\vert_{\partial\D\setminus\left([a_0a_1]\cup[a_1a_2]\right)}\ ,
\]
for any $\tau$. 

Now we analyse the flux of $v_\tau$ on $E_\tau$. We have, 
\[
2H\area(E_\tau^T)=\absc{a_0b_1}-\absc{b_1b_2(\tau)}+\absc{b_2(\tau)a_1}+F_{v_\tau}(\gamma)+\int_{[a_0a_1]}\langle X_\tau,-\nu \rangle ds,
\]
where $\gamma$ are the arcs of horocycles inside $E_\tau$, and $E_\tau^T$ is the domain $E_\tau$ truncated by the horocycles. Taking a limit of a sequence of nested horocycles going to the vertices of $E_\tau$, using  \eqref{eqa21}, we obtain

\begin{eqnarray}
\nonumber 2H\area(E_\tau^T)-2H\area(E_\tau)&=&-\varphi(\tau)+\int_{[a_0a_1]}(1-\langle X_\tau, \nu\rangle)ds +F_{v_\tau}(\gamma)\\[5pt]
\label{eqa23} \Longrightarrow 0&=&-\varphi(\tau)+\int_{[a_0a_1]}(1-\langle X_\tau, \nu\rangle)ds.
\end{eqnarray}

Similarly, the flux of $v_\tau$ on $E^\prime_\tau$, gives
\begin{eqnarray}
\nonumber 2H\area(E^{\prime\  T}_\tau)&=&-\absc{a_1b_3(\tau)}+\absc{b_3(\tau)b_4}-\absc{b_4a_2}+F_{v_\tau}(\gamma)+\int_{[a_1a_2]}\langle X_\tau,-\nu\rangle ds   \\[5pt]
\nonumber 2H\area(E^{\prime\  T}_\tau) - 2H\area(E^{\prime\  T}_\tau)&=& \varphi(\tau)-\int_{[a_1a_2]}\left( 1+\langle X_\tau, \nu\rangle\right)  ds +F_{v_\tau}(\gamma).
\end{eqnarray}
Taking the limit for a sequence of nested horocycles, we obtain 
\begin{equation}
\label{eqa24}
\varphi(\tau)= \int_{[a_1a_2]}\left( 1+\langle X_\tau, \nu\rangle\right)  ds.
\end{equation}

Then for any family of disjoint arcs $\rho$ of the boundary $\partial \D$, we have
\begin{equation}
\label{eqa25} \left\vert\int_{\rho}\langle X-X_\tau, \nu\rangle ds\right \vert \leq  \int_{[a_0a_1]\cup[a_1a_2]}\left\vert \langle X-X_\tau, \nu\rangle\right\vert  ds \leq 2\varphi(\tau),
\end{equation} 
since $\varphi(\tau)$ tends to zero when $\tau$ tends to zero, we proved \eqref{eqa27} for points on $\partial\D$.

Let $p$ be a point in the interior of $\D$. Let $\alpha$ be the level curve of $v_\tau-u$ through $p$. This level curve  goes to the vertices of $\partial\D$.

Let $\Sigma$ be the graph of $u$ in $\D$ and $\Sigma_\tau$  the graph of $v_\tau$ in $\D_\tau$. $\Sigma$ and $\Sigma_\tau$ are complete and stable, so by curvature estimates \cite{RST}, for any $\mu>0$, there exist $\rho>0$ (independent of $\tau$) such that for all $p$ in $\D$, if $q_1\in\Sigma_\tau\cap B((p,v_\tau(p)), \rho) $ and $q_2\in\Sigma\cap B((p,u(p)), \rho), $  then 
\begin{equation}
\label{eqa26}
\|n_\tau(p)-n_\tau(q_1)\|\leq \mu
 \hspace{1cm} \textnormal{and}\hspace{1cm}
 \|n(p)-n(q_2)\|\leq \mu,
 \end{equation}
where $n_\tau(q)$ is the downwards unit normal vector to $\Sigma_\tau$ at $q$, $n(q)$ is the downwards unit normal vector to $\Sigma$ at $q$, $B((q,t), \rho)$ is the ball in $\mathbb{H}\times\real$ centered at $(q,t)$ having radius $\rho$. 

Let us fix a $\mu>0$ and $p\in\D$. Then, there is $0<\rho_1<\dfrac{\rho}{2}$ which does not depend on $\tau$, such that

\begin{equation} \label{j3}
 \vert u(q)-u(p)\vert <\dfrac{\rho}{2} \hspace{.7cm} \textnormal{for every} \hspace{.7cm} q\in D(p,\rho_1),\end{equation}
where $D(p,\rho_1)$ is the disk in $\mathbb{H}^2$ centered at $p$ having radius $\rho_1$.
 
\begin{claim}\label{claim10}
If $\|n(p)-n_\tau(p)\|\geq 3\mu$ \ \  then\ \  $\varphi(\tau)\geq \dfrac{\rho_1\mu^2}{4}$.
\end{claim}

Let us assume that Claim \ref{claim10} has been proved. Let $\tau$ be a positive real number sufficiently small, such that $\varphi(\tau)<\dfrac{\rho_1}{4}$. Then, by Claim  \ref{claim10}, we have $\|n(p)-n_\tau(p)\|<3\mu$. On the other hand, using  Lemma \ref{lemmaapx}, we  obtain, 
$$\|X(p)-X_\tau(p)\|\leq\|n(p)-n_\tau(p)\|\leq3\mu,$$
which give us the desired behaviour for a small $\tau$. So after a translation, so that  $v_\tau(p_0)=u(p_0)$ for a fixed $p_0\in \D$, we obtain  $\displaystyle\lim_{\tau\to 0}v_\tau\vert_\D=u.$
 
Let us prove the Claim \ref{claim10}. 
\begin{proof} [Proof of the Claim \ref{claim10}.] Assume that $\|n(p)-n_\tau(p)\|\geq 3\mu$.  Let $\Omega_\tau(p)$ be the connected component of $\{u-v_\tau(p)>u(p)-v_\tau(p)\}$ which contains $p$ in its boundary. We denote by $\Lambda_\tau$ the connected component of $\partial\Omega_\tau(p)$ that contains $p$, we have $\Lambda_\tau$ piecewise smooth since it is a level curve of $u-v_\tau(p)$.

The image of $\Lambda_\tau\cap D(p,\rho_1)$ under $u$ and $v_\tau$ are two parallel curves $\sigma\subset\Sigma$ and $\sigma_\tau\subset\Sigma_\tau$, respectively. By  \eqref{j3},  for any $q\in\Lambda_\tau\cap D(p,\rho_1)$, we have  $(q,u(q))\in \sigma$ and
\begin{equation}
\label{j6}
\|(q,u(q))-(p,u(p))\|\leq \rho_1+\rho/2\leq \rho, 
\end{equation} 
since $\rho_1\leq \rho/2$.
 
Using the curvature estimates \eqref{eqa26}, inequality \eqref{j6}, implies
\begin{equation}
\label{j7} \|n(q)-n(p)\|\leq \mu. 
\end{equation}

Similarly, for any $q\in\Lambda_\tau\cap D(p,\rho_1)$, we have $(q, v_\tau(q))\in\sigma_\tau$ and $u(q)-v_\tau(q)=u(p)-v_\tau(p)$ which implies $v_\tau(q)-v_\tau(p)=u(q)-u(p)$. Then, by \eqref{j6} and \eqref{j3}
\begin{equation}
\label{j9} 
\|(q,v_\tau(q))-(p,v_\tau(p))\|\leq\rho_1+\rho/2\leq \rho.
\end{equation}
Using curvature estimates \eqref{eqa26}, we obtain
\begin{equation}
\label{j10}
\|n_\tau(q)-n_\tau(p)\|\leq\mu.
\end{equation}

Then, using \eqref{j7}, \eqref{j10} and the hypothesis
\begin{eqnarray}
\nonumber\|n(q)-n_\tau(q)\|&=&\|n(q)-n_\tau(p)+n_\tau(p)-n(p)+n(p)-n_\tau(q) \|\\[5pt]
\nonumber&=&\|(n(q)-n(p))+(n(p)-n_\tau(p))+n_\tau(p)-n_\tau(q)\|\\[5pt]
\nonumber&\geq & \|n(p)-n_\tau(p)\|-2\mu\\[5pt]
\label{j11}&>&\mu. 
\end{eqnarray}

This last inequality \eqref{j11} and Lemma \ref{lemmaapx}, gives
\begin{equation}
\label{j12} 
\int_{\Lambda_\tau\cap D(p,\rho_1)}\langle X-X_\tau, \eta\rangle ds\geq \int_{\Lambda_\tau\cap D(p,\rho_1)}\dfrac{\|n_\tau-n\|^2}{4}ds\geq \dfrac{\rho_1\mu^2}{2}.
\end{equation}

If $\Lambda_\tau$ is not compact, $\Lambda_\tau$ goes to two vertices of $\partial\D$. There is a compact arc $\Phi\subset\Lambda_\tau$ and two small arcs $\gamma\subset\D$ and $\widetilde{\gamma}\subset \D$ joining the extremities of $\Phi$ to an arc (maybe disjoint) $\Psi$  in $\partial\D$. Denoting $\eta=\dfrac{\nabla (u-v_\tau)}{\|\nabla (u-v_\tau)\|}$,  using \eqref{j12}, \eqref{eqa25} and choosing $\gamma$ and $\widetilde{\gamma}$ small enough, the flux of $u-v_\tau$ gives 
\begin{eqnarray*}
0&=&\int_{\Phi}\langle X-X_\tau,-\eta\rangle ds +\int_{\Psi}\langle X-X_\tau,\nu\rangle ds+F_{u-v_\tau}(\gamma\cup\widetilde{\gamma})\\[5pt]
&\leq& \dfrac{-\rho_1 \mu^2}{2}+2\varphi(\tau)+  F_{u-v_\tau}(\gamma\cup\widetilde{\gamma}), 
\end{eqnarray*}
which implies,
$$\varphi(\tau)\geq \dfrac{\rho_1\mu^2}{4}, $$
as claimed.

\end{proof}

The proof of Claim  \ref{claim10} concludes the proof of Lemma \ref{lemaconvergencia}.

\end{proof}

Now, we present a technical lemma used in the proof of Lemma \ref{lemaconvergencia}.

\begin{lemma}\label{lemmaapx} Let $u$ and $u^\prime$ be two solutions of the Dirichlet problem in $\Omega\subset\mathbb{H}^2$ and  $n$, $n^\prime$  their downward pointing unit normals. Then, at any regular point of $u-u^\prime$
\begin{equation}\label{eqa30}
\langle X^\prime- X, \eta\rangle_{\mathbb{H}^2}\geq \dfrac{\|n^\prime-n\|^2}{4}\geq \dfrac{\|X^\prime-X\|}{4},
\end{equation} 
 where $X$ and $X^\prime$ are the projection of $n$ and $n^\prime$ on $\mathbb{H}^2$, respectively; and $\eta=\dfrac{\nabla (u^\prime-u)}{\|\nabla (u^\prime-u\|)}$ orients the level curve.
\end{lemma}

The proof of this lemma is analogous to the proof of  Lemma A.1 in \cite{CR}.  

%%%%%%%%%%%%%%%%%%%%%
%%%%%%%%%%%%%%%%%%%%%

\section{Conformal type} \label{conformal}

In \cite{MN} the authors make a complete study of the area growth of H-graphs. In particular they prove that a complete Scherk H-graph has quadratic area growth \cite[Theorem 4]{MN}. We will now give another proof of this.

Let $\Sigma$  be a constant mean curvature $H$  Scherk graph, 0<H<1/2,  of a function $u$ over an admissible domain $\D$.   We denote by $D^{\Sigma}(p, R)$ the intrinsic radius $R$ disc of $\Sigma$ centered at $p$, and by $D^\D(x,R)$ the intrinsic radius $R$ disc of $\D$ centered at $x$. Observe that if $p=(x,t)\in\mathbb{H}^2\times\real$, then $\pi(D^{\Sigma}(p,R))\subset D^\D(x, R)$, where $\pi$ denotes the horizontal projection in the first factor. 

 \begin{proposition} \label{lemmaconformal} Let $p_0=(x_0,0)$ be a point in $\Sigma$. Then, there is a constant $C>0$ and $r_0>0$  such that 
 $$\area_{\Sigma}(D^\Sigma(p_0, r))\leq C\, r^2, \textnormal{ \ \ for \ \ } r>r_0, $$
where $\area_{\Sigma}$ is the intrinsic area of $\Sigma$. In particular, $\Sigma$ is parabolic.
 \end{proposition}

\begin{proof} 
We prove that a constant mean curvature Scherk graph has quadratic area growth and consequently its conformal type is $\mathbb{C}$. Let $\Sigma$  be a constant mean curvature $H, 0<H<1/2$ Scherk graph of a function $u$ over an admissible domain $\D$ and let $p_0=(x_0,0)\in\hr$ be a point in $\Sigma$.

Given a sequence of points $\{p_n\}$ in $\D$ converging to a point $q$ in the boundary of $\D$, we have, by curvature estimates \cite{RST} that there exist a $\delta>0$ which  does not depend on $n$, such that a neighbourhood $\Sigma(p_n,\delta)$ of $p_n$ in $\Sigma$ is a graph of bounded geometry over a disc centered at the origin of $T_{p_n}\Sigma$ of radius $\delta$. We denote by $G(p_n,\delta)$ the graph $\Sigma(p_n,\delta)$ translated vertically by $-u(p_n)$. So, one can show that $G(p_n,\delta)$ converges to a subset of a disc in $\partial\D\times\real$ having radius $\delta^\prime$ centered at $q$.  Because of this convergence,  we say that $\Sigma$ converges uniformly on compact subsets to $\partial\D\times\real$. Let  $T_\epsilon(\partial\D)$ be the $\epsilon$-tubular neighbourhood of $\partial\D$ contained in $\D$. Given $\epsilon>0$, there is $R>0$, such that $\pi(D^\Sigma(p_0, R))\cap T_\epsilon(\partial\D) $ is an annulus. We fix disjoint horocycles $\hor_i$ at vertices  of $\partial\D$, such that the $F_i\cap\D$ contains points outside $T_\epsilon(\partial\D)$. The boundary of  $\pi(D^\Sigma(p_0, R))$ is composed of arcs $\gamma_i^R$ in $F_i$ and arcs $\eta_i^R$ joining two disjoint horocycles.  
 
 Now let $r>0$.  We want to estimate the area of  $D^{\Sigma}(p_0,R+r)\setminus\D^{\Sigma}(p_0,R)$. First, let us estimate the area of $\left(D^{\Sigma}(p_0,R+r)\setminus\D^{\Sigma}(p_0,R)\right)\setminus \left( \bigcup_{i} F_i\times\real\right)$. We observe that the arcs $\eta^{R+r}_i$ in the $\pi(\partial D(p_0, R+r))$  converge uniformly to $\left(\partial\D\right)\setminus\left(\bigcup_i F_i\right)$. Then, since $\Sigma$ converges to $\partial\D\times\real$, the growth of $\left(D^{\Sigma}(p_0,R+r)\setminus\D^{\Sigma}(p_0,R)\right)\setminus\left(\bigcup_i F_i\times\real\right)$ is at most linear in $r$. Now, let us analyse the area growth of $\left(D^{\Sigma}(p_0,R+r)\setminus\D^{\Sigma}(p_0,R)\right)\bigcap\left(\bigcup_i F_i\times\real\right)$. Fix a vertex $i_0$, the curve $\gamma^{R+r}_{i_0}$ converges to $(A_{i_0}\cup B_{i_o})\cap F_{i_0}$, where $(A_{i_0}\cup B_{i_o})=\partial\D\cap F_{i_0}$. Then, the length of  $\gamma^{R+r}_{i_0}\setminus \gamma^{R}_{i_0}$ grows, at most,  linearly. Moreover, if $M:=\sup_{p\in \pi(D^{\Sigma}(p_0, R))}\vert u(p)\vert $, we have 
$$u\left(\left(D^{\Sigma}(p_0,R+r)\setminus\D^{\Sigma}(p_0,R)\right)\bigcap\left(\bigcup_i F_i\times\real\right)\right)\subset [-M-r, M+r],$$ 
 which implies that 
 $$\left(D^{\Sigma}(p_0,R+r)\setminus\D^{\Sigma}(p_0,R)\right)\bigcap\left(\bigcup_i F_i\times\real\right)$$ has, at most,  quadratic area growth in $r$.  Thus, we conclude that $\left(D^{\Sigma}(p_0,R+r)\setminus\D^{\Sigma}(p_0,R)\right)$ has quadratic area growth, that is, there is a constant $C>0$, such that
 \begin{equation}\label{eq30} \area\left(D^{\Sigma}(p_0,R+r)\setminus\D^{\Sigma}(p_0,R)\right)\leq C r^2,\ \  \  \textnormal{when}\ \ \  r\to\infty 
 \end{equation}
 
It remains to estimate the area of $D^{\Sigma}(p_0, R)$, we will show that it is finite. We have $\pi(D^{\Sigma}(p_0, R))\subset D^{\D}(x_0, R)$ and $D^{\Sigma}(p_0, R)\subset \s $, where $\s$ is the region of $\hr$ inside the cylinder $\mathcal{C}=\pi^{-1}(\pi(D^{\Sigma}(p_0, R)))$ which is below $D^{\Sigma}(p_0, R)$ and above $D_2:=\mathcal{C}\cap (\mathbb{H}^2\times\{-M\} )$, where $M=\sup_{p\in \pi(D^{\Sigma}(p_0, R))}\vert u(p)\vert $. Observe that $D^{\Sigma}(p_0, R)$ foliates $\s$ by surfaces having constant mean curvature $2H$. Let $\overrightarrow{n}  $  be the unit normal vector to $\Sigma$ pointing up.  Denoting by $D_1:=D^{\Sigma}(p_0,R)$ and $U:=\partial\s\setminus(\,D_1\,\cup\,D_2\,)$, using Stokes Theorem, we obtain
 
 \begin{eqnarray}
\nonumber 2H\mathrm{Vol}(\s)&=&\int_{D_1}\langle \overrightarrow{n}, \nu_{D_1}\rangle ds+\int_{D_2}\langle \overrightarrow{n}, \nu_{D_2}\rangle+\int_{U}\langle \overrightarrow{n}, \nu_{U}\rangle ds\\[5pt]
%%%%%%%%%%%%%%%%%%%%%%% 
 \nonumber &\geq& \area(D_1)-\area(D_2)-\area(U)\\[5pt]
%%%%%%%%%%%%%%%%%%%%%%
\nonumber 
 \Longrightarrow  \area(D_1)&\leq &  2H\mathrm{Vol}(\s)+\area(D_2)+\area(U)\\
%%%%%%%%%%%%%%%%%%%%%%
\label{eq31}
\area(D_1) &\leq& {\bf C},
 \end{eqnarray}
  where ${\bf C}$ is a  positive constant depending on $R$. So the area of $D^{\Sigma}(p_0,R)$ is finite. 
  
 Choosing  $r_0>R>1$, by \eqref{eq30} and \eqref{eq31},  we have 

$$\area_{\Sigma}(D^\Sigma(p_0, r))\leq C\, r^2, \textnormal{ \ \ for \ \ } r>r_0. $$
 
  \end{proof}

%%%%%%%%%%%%%%%%%%%%%%%%%%%%%
%%%%%%%%%%%%%%%%%%%%%%%%%%
%%%%%%%%%%%%%%%%%%%%%%%%%

\section{Main Theorem} \label{main}

In this section using the extension process described in Section \ref{extension} we construct an entire H-graph which is conformally  the complex plane and whose ideal boundary is $(\partial_\infty\mathbb{H}^2)\times\real$.

\begin{theorem} \label{main1}
For each $H, 0\leq H<1/2$,  there is a parabolic entire  H-graph in $\hr$ whose asymptotic boundary is $(\partial_\infty \mathbb{H}^2)\times\mathbb{R}$.
\end{theorem}

\begin{proof} 
Let $\D_0$ be an admissible domain and $u_0$ be a solution of the Dirichlet problem over $\D_0$. The graph of $u_0$ is parabolic so we can find compact disks $K_0\subset K_1$ in $\D_0$ such that the conformal modulus of the annulus in the graph of $u_0$ over $K_1\setminus\mathrm{Int}(K_0)$ is greater than one. 

For $\tau>0$ sufficiently small, we can extend $\D_0$ to a admissible domain $\D_1(\tau)$ with a solution of the Dirichlet problem $u_1=u_1(\tau)$ over $\D_1$ such that $\|u_1-u_0\|_{C^2(K_1)}$ is as small as desired. In particular, so that graph of $u_1$ over $K_1-\mathrm{Int}(K_0)$ is also of conformal modulus greater than one. $\D_1$ is constructed by Lemma \ref{lemma10} attaching curved quadrilaterals $E_\tau$ and $E_{\tau}^\prime$ to each pair of sides of $\partial\D_0$. By Lemma \ref{lemaconvergencia} we know $u_1=u_1(\tau)$ can be chosen as close to $u_0$ on $K_1$ as we wish. The graph of $u_1$ is parabolic so we can choose $K_2\subset \D_1= \D_1(\tau)$ such that the graph of $u_1$ over the annuli  $K_2\setminus\mathrm{Int}(K_1)$  and $K_1\setminus\mathrm{Int}(K_0)$ has  conformal modulus greater than one, if $\|u_1-u_0\|_{C^2(K_1)}<\epsilon_1$, for some $\epsilon_1>0$.

To construct the entire graph we proceed by induction. Let $\epsilon_i,  i\geq 1$ be chosen so that $\displaystyle\sum_{i=1}^{\infty}\epsilon_i<\infty$.  Assume we have constructed $n+1$ admissible domains $\D_0, \D_1, \cdots, \D_n$ and compact disks $K_0, K_1, \cdots,$ $ K_{n+1}$, and solutions of the Dirichlet problem $u_j$ over $\D_j, 0\leq j\leq n$ such that 
\begin{enumerate}
[(i)]
\item $\|u_n-u_{n-1}\|_{C^2(K_n)}<\epsilon_n$.
\item each $K_j-\mathrm{Int}(K_{j-1}), 1\leq j\leq n+1$ is an annulus.
\item the conformal modulus of the annulus in the graph of $u_n$ over $K_j-\mathrm{Int}(K_{j-1})$ is greater than one, for $ 1\leq j\leq n+1$.
\end{enumerate}

Now it is clear how to extend $(\D_n, u_n, K_j)$ to $(\D_{n+1}, u_{n+1}, K_j)$ so (i), (ii), (iii) are satisfied. This is done exactly as we did to go from $(\D_0, u_0, K_0, K_1)$ to $(\D_1, u_1, K_0, K_1, K_2)$. We do not repeat this. 

Next we prove the compact sets exhaust $\mathbb{H}^2$. In fact some more care in their choice is necessary to assure this. Consider $K_0, K_1, K_2$. As before $K_0$ is a fixed compact disk in $\D_0$. $K_1$ is chosen so that $u_0$ had conformal modulus greater than one in $K_1\setminus\mathrm{Int}(K_0)$. Let $\eta_1>0$ and enlarge $K_1$ in $\D_0$ to a compact disk $\widetilde{K}_1\subset\D_1$ such that $\mathrm{dist}(\partial\widetilde{K}_1, \partial\D_0)<\eta_1$.  Clearly the conformal modulus of $u_0$ on $\widetilde{K}_1\setminus\mathrm{Int}(K_0)$ is greater than one. Relabel $\widetilde{K}_1=K_1$. Let $\D_2\supset \D_1$ be an admissible domain as before and $K_2$ as well. Enlarge $K_2$ to $\widetilde{K}_2$ so that $\mathrm{dist}(\partial\widetilde{K}_2, \partial\D_1)<\eta_2$. Relabel $\widetilde{K}_2=K_2$.  In general, we enlarge $K_n$ to $\widetilde{K}_n$ so that the distance of $\partial\widetilde{K}_n$ to $\partial\D_{n-1}$ is less than $\eta_n$ and $\displaystyle\lim_{n\to\infty} \eta_n=0$. Relabel $\widetilde{K}_n=K_n$.

Now we have to prove that the sequence of compacts $\{K_n\}$ exhaust $\h$.  We observe that, by Claim  \ref{claim1} as long as $d_3^*$ is larger than $\dfrac{2\mu^*}{1-4H^2} + 1$, there exists a constant $c>0$ such that the distance between  $B_1$ and $B_2$ and the distance between $A_1$ and $A_2$ is larger than $c$. So the boundary of $\D_{n+1}$ is a constant farther from $\partial\D_{n}$. Then, $\partial\D_n$ diverges to infinity when $n$ tends to infinity. Since we choose $K_n$ such that $\partial K_n$ is  close to the boundary of $\partial\D_n$,  $\{K_n\}$ exhausts $\h$. 

To obtain the entire graph, we let $n$ tend to infinity, since $\{u_n(x)\}$ is a Cauchy sequence for any $x\in\h$ we obtain a map $u$ which give us a constant mean curvature graph over $\h$. Moreover, since $u_n$ converges uniformly to $u$ on each $K_{j+1}-\mathrm{Int}(K_{j})$ and its graph has modulus at least one over $K_{j+1}-\mathrm{Int}(K_{j})$,  the modulus of the graph $u$ over $K_{j+1}-\mathrm{Int}(K_{j})$ is at least one. Consequently, by Gr\"{o}tzsch Lemma  \cite{V},  the conformal modulus of the graph of $u$ is $\mathbb{C}$.
 
 \end{proof}

%\bibliographystyle{acm}
%\bibliography{Ref}

\end{document}